%% file: main.tex
\documentclass{amsart}
\usepackage{geometry}
\usepackage{graphicx}
\usepackage{hyperref}
\usepackage{bm} %
\usepackage{amssymb,amsmath,amsthm}
\usepackage[fixamsmath,disallowspaces]{mathtools} %
\mathtoolsset{showonlyrefs=true}
\usepackage{nicefrac}       %
\usepackage{microtype}      %
\usepackage{algorithm, algpseudocode}
\usepackage{comment}
\usepackage{enumitem}
\usepackage{ulem}  %
\usepackage{tikz}
\usetikzlibrary{shapes}
\usepackage{pgfplots}
\usepackage{pgfplotstable}
\pgfplotsset{compat=newest}
\usepackage[dvipsnames]{xcolor}
\usepackage{accents}
\newcommand{\ubar}[1]{\underaccent{\bar}{#1}}

\theoremstyle{plain}
\newtheorem{theorem}{Theorem}
\newtheorem{lemma}[theorem]{Lemma}

\theoremstyle{definition}
\newtheorem{definition}[theorem]{Definition}
\newtheorem{remark}[theorem]{Remark}

\DeclareMathOperator*{\supp}{supp}

\makeatletter
\renewcommand{\paragraph}{\@startsection{paragraph}{4}{\z@}%
  {1ex \@plus1ex \@minus.2ex}%
  {-0.5em}%
  {\normalfont\normalsize\scshape}}
\makeatother

\title{Multilevel Sparse Tensor Approximation for High-Dimensional Parametric PDEs}
\author{Martin Eigel, Philipp Trunschke, Dana Wrischnig}
\date{\today}

\begin{document}

\begin{abstract}
    \input{content/abstract}

\end{abstract}


\maketitle

\input{content/introduction}
\input{content/theory}

\input{content/numerics}

\section*{Acknowledgements}
This project is partially funded by the ANR-DFG project COFNET (ANR-21-CE46-0015). DW acknowledges funding by the Deutsche Forschungsgemeinschaft (DFG, German Research Foundation) through the project Robust Sampling for Bayesian Neural Networks Project ID 522337282 and CRC/TRR 388 Rough Analysis, Stochastic Dynamics and Related Fields Project ID 516748464.

Our code makes extensive use of the Python packages \texttt{numpy}~\cite{numpy}, \texttt{scipy}~\cite{scipy}, \texttt{matplotlib}~\cite{matplotlib} and \texttt{FEniCSx}~\cite{fenics_dolfinx,fenics_basix_1,fenics_basix_2,fenics_ufl}.

\bibliographystyle{siam}
\bibliography{references}

\appendix

\input{content/appendix/tensors}
\input{content/appendix/proof_summability_2}
\input{content/appendix/proof_rip_transfer}

\end{document}

%% file: content/abstract.tex
In this paper the efficiency of multilevel sparse tensor approximation methods for high-dimensional affine parametric diffusion equations is investigated. Methodologically, the recently presented Sparse Alternating Least Squares (SALS) algorithm is employed to construct adaptive tensor train (TT) approximations of quantities of interest (QoI).
By combining this tensor-based approach with a multilevel Galerkin discretization strategy, the solution's regularity can be exploited to significantly reduce computational costs by level-adapted sample sizes.
A rigorous theoretical analysis is derived, demonstrating that the work overhead for the proposed multilevel method remains independent of the discretization level, which stands in stark contrast to the exponential growth observed in single-level approaches.
The presented analysis is quite general and not constrained to the sparse TT format but uses a generic framework that can be extended to other model classes.
Numerical experiments validate the predicted efficiency gains in high-dimensional settings.

%% file: content/introduction.tex
\section{Introduction}

Efficient numerical treatment of partial differential equations (PDEs) with high-dimensional parameter spaces remains a formidable challenge in uncertainty quantification and computational engineering.
This work addresses this challenge by combining multilevel Galerkin discretizations with recently presented adaptive sparse tensor network approximations.
By exploiting both, the hierarchical structure of the spatial discretization and the low-rank nature of the parametric solution, we derive a method for computing quantities of interest that offers rigorous error control while mitigating the curse of dimensionality.

\subsection{Motivation}

High-dimensional parametric partial differential equations arise throughout uncertainty quantification (UQ), where uncertain inputs (e.g., diffusion or permeability coefficients) are modeled as random fields~\cite{GhanemSpanos1991,XiuKarniadakis2002}.
Using for instance the common truncated Karhunen-Loève expansion, fields with bounded variance can be modeled as deterministic functions of random parameters.
The Doob--Dynkin lemma ensures that the solution field can also be written as a function of these parameters.
A principal difficulty is that naive tensor-product discretizations in the parameter domain exhibit computational complexity that grows exponentially with the number of parameters $M$, the notorious curse of dimensionality (CoD).
Sparse polynomial and sparse-grid collocation methods~\cite{CohenDeVore2015,NobileTemponeWebster2008} mitigate this effect by exploiting anisotropic regularity and coefficient sparsity.
However, constructing high-accuracy surrogates can still require many expensive forward solves on fine spatial meshes.
A well-known strategy to alleviate this effort is to use a multilevel approach with level-dependent sample sizes, cf.~\cite{TeckentrupJantschWebsterGunzburger2015,Heinrich2001,Giles2008,cliffe2011multilevel,teckentrup2013further,harbrecht2012multilevel}.

In many applications, the computational objective is the efficient evaluation of quantities of interest (QoIs), rather than the full solution.
The compressed sensing (CS) analyses in~\cite{RauhutSchwab2017_CSPG} is concerned with this task in the context of affine-parametric operator equations. They show near-best $s$-term convergence under weighted $\ell^p$-summability of coefficient sequences.
The multilevel variant~\cite{Bouchot2017} further reduces the work by distributing samples across multiple spatial discretization levels.
Building on this idea,~\cite{TrunschkeEigelNouy2025} proposes to encode the weighted-sparse model classes in tensor networks with sparse component tensors, with the aim of speeding up the computation.
In the spirit of multilevel methods~\cite{Giles2008,Heinrich2001,TeckentrupJantschWebsterGunzburger2015}, this work combines the new sparse tensor train (TT) approximation proposed in~\cite{TrunschkeEigelNouy2025} with the multilevel Galerkin hierarchy proposed in~\cite{Bouchot2017}, thus allocating most samples to inexpensive coarse levels.

The key methodological point is to combine an \emph{adaptive sparse TT least-squares estimator} (SALS/SSALS) with a \emph{multilevel Galerkin hierarchy} in the physical variables, so that most samples are taken on inexpensive coarse levels while the parametric dependence is captured in a compressed tensor network model class.

\subsection{Related work and contributions}\label{sec:related-work-contrib}

High-dimensional parametric PDEs are commonly approximated via polynomial surrogates constructed by sparse grids and sparse interpolation/collocation techniques, exploiting anisotropy and mixed regularity to mitigate the exponential complexity of full tensor-product discretizations, see for instance~\cite{BungartzGriebel2004,CohenDeVore2015,NobileTemponeWebster2008}.
A complementary line of work builds sparse polynomial-chaos representations from random solution samples using compressed sensing (CS), typically based on (weighted) $\ell_1$-minimization, with recovery guarantees under weighted summability assumptions on coefficient sequences and suitable sampling measures~\cite{DoostanOwhadi2011_NonadaptedSparse,RauhutWard2016_WeightedL1,RauhutSchwab2017_CSPG,FoucartRauhut2013_CSBook}.
Multilevel variants of CS Petrov--Galerkin and stochastic collocation combine these approximation principles with spatial discretization hierarchies to reduce the overall cost by allocating most samples to inexpensive coarse levels~\cite{Giles2008,Heinrich2001,TeckentrupJantschWebsterGunzburger2015,Bouchot2017}.

Tensor network formats such as the popular tensor train (TT) representation provide an alternative approximation architecture for high-dimensional coefficient tensors, enabling storage and evaluation costs that scale polynomially in the dimension for structured problems \cite{Oseledets2011,Hackbusch2012,GrasedyckKressnerTobler2013}.
In the context of UQ and parametric PDEs, TT-based surrogates have been developed for polynomial expansions~\cite{Eigel2020a,Eigel2020b} and regression-type constructions~\cite{eigel2019variational}.
The present work builds specifically on sparse TT least-squares estimators (SALS/SSALS) grounded in weighted sparsity model classes, which adaptively identify and fit a sparse TT representation from pointwise evaluations, cf.\ \cite{TrunschkeEigelNouy2025} and references therein.

\vspace{1ex}

This work develops a comprehensive framework for multilevel approximation methods.
Our primary contributions are threefold:
\begin{enumerate}
    \item \textbf{General work--accuracy analysis.}
    We derive work--accuracy relations for both single-level and multilevel surrogate computations.
    In particular, we show that the multilevel method yields an overhead (relative to the finest-level solve cost) that is independent of the finest spatial level up to logarithmic factors.
    By contrast, the single-level computational overhead grows exponentially as the accuracy target tightens.
    While the resulting relations are not new, our generalised analysis cleanly exposes the underlying assumptions \ref{asn:level_error}--\ref{asn:work_bound}.
    While \ref{asn:level_error} and \ref{asn:work_bound} depend only on the spatial discretisation and the PDE solution's regularity, \ref{asn:approx_error_eps} and \ref{asn:emp_approx_error} concern the least squares approximation in the model class.
    This clear separation of responsibilities together with the generality of the assumptions simplify the transfer of the presented theory to other PDEs and model classes.

    \item \textbf{Application to sparse tensor trains and a parametric diffusion problem.}
    For a canonical affine-parametric diffusion problem, we relate assumptions \ref{asn:level_error}--\ref{asn:work_bound} to standard PDE approximation properties, parametric regularity, and conditions on the model class.
    Extending the theory presented in~\cite{TrunschkeEigelNouy2025} and~\cite{Bouchot2017}, we prove \ref{asn:level_error} and \ref{asn:work_bound} for uniformly refined FEM.
    We believe that these results can be easily transferred to other PDEs with analytic parameter to solution maps.
    We then propose to approximate the QoI in a model class of tensor train (TT) tensors with sparse component tensors and prove assumptions \ref{asn:approx_error_eps} and \ref{asn:emp_approx_error}.
    We believe that these proves can be transferred to other model classes, for example using orthogonal polynomial bases for which conversion rules to monomials are known (like Laguerre polynomials).

    \item \textbf{Empirical evaluation.}
    We present numerical experiments that quantify the practical work--accuracy behavior of the proposed multilevel method and compare SALS and SSALS~\cite{TrunschkeEigelNouy2025} under matched accuracy targets. The experiments corroborate the predicted trends and illustrate the computational gains that can be achieved by the proposed multilevel method.
\end{enumerate}

\section{Setting}

We consider a parametric partial differential equation with homogeneous boundary conditions on a bounded Lipschitz domain $D \subset \mathbb{R}^d$.
The problem representation, notation and basic solvability results are based on formulations in~\cite{Bouchot2017, Cohen2010, RauhutSchwab2017_CSPG, Todor2006}.

For any parameter vector $\bm{y}\in U := [-1, 1]^{M}$, $M\in\mathbb{N}$,~$M>0$, the differential equation reads
\begin{equation}
\label{eq:darcy}
\begin{aligned}
    - \nabla_x \cdot (a(x, \bm{y}) \nabla_x u(x, \bm{y})) &= f(x)
    \quad\text{for}\ x \in D\,, \\
    u(x, \bm{y}) &= 0
    \quad\text{for}\ x\in \partial D\,,
\end{aligned}
\end{equation}
with a parametric \emph{diffusion coefficient}
\begin{equation}
\label{eq:diffusionCoefficientReal}
    a : D\times U \to \mathbb{R},
    \quad
    a(x, \bm{y}) = a_0(x) + \sum_{m=1}^M \bm{y}_m a_m(x) \,.
\end{equation}
Such an expansion is given, for example, by the Karhunen--Loève decomposition of random fields~\cite[Theorem 2.5]{Schwab2006}.
Let $V := H^1_0(D)$ and assume that $f \in V^*$ and the coefficient $a$ satisfies the \emph{uniform ellipticity assumption}
\begin{equation}
\label{eq:UEA}
    r \le a(x, \bm y) \le R
    \quad\text{for all}\quad
    (x,\bm y)\in D\times U
    \,,
\tag{UEA}
\end{equation}
\noeqref{eq:UEA}%
then the Lax--Milgram theorem ensures well-posedness of the variational formulation
\begin{equation}
\label{eq:variational_form}
    \mathcal{A}(u(\bm{y}), v; \bm{y})
    = \langle f, v\rangle_{V^*, V}
    \quad\text{for all}\quad
    v\in V,
\end{equation}
with the parametric bilinear form $\mathcal{A}(u, v; \bm{y}) := \int_D a(\bm{y}) \nabla u(x) \nabla v(x) \,\mathrm{d}x$.

Often, one is not interested in the full solution $u(\bm{y}) := u(\bullet, \bm{y})\in H^1_0(D)$ but rather in a \emph{quantity of interest} (QoI)
$$
    g: U \to \mathbb{R}
    \qquad\text{with}\qquad
    g(\bm{y}) := G(u(\bm{y})),
$$
where $G\in V^*$ is a bounded linear functional.

To approximate $g(\bm{y})$ numerically, let $(V_l)_{l\in\mathbb{N}}$ be a sequence of nested, finite-dimensional subspaces of $V$ that are dense in $V$, i.e.,
\begin{equation}
    \label{eq:galerkinSpaces}
    V_0 \subset V_1 \subset \dots
    \qquad \text{and} \qquad
    \operatorname{cl}_V\big({\textstyle\bigcup_{l\in\mathbb{N}} V_l}\big) = V~.
\end{equation}
Restricting~\eqref{eq:variational_form} to the discretisation space $V_l$ yields for every $\bm{y}\in U$ a solution $u_l(\bm{y})\in V_l$ satisfying the finite-dimensional linear system
\begin{equation}
\label{eq:variationalFE}
    \mathcal{A}(u_l(\bm{y}), v_l; \bm{y})
    = (f, v_l)_{L^2(\rho)}
    \quad\text{for all}\quad
    v_l\in V_l\,.
\end{equation}
The corresponding QoI is given by
\begin{equation}
\label{eq:gl}
    g_l(\bm{y}) := G(u_l(\bm{y})) \,.
\end{equation}

Naïvely evaluating $g_l(\bm{y})$ requires the solution $u_l(\bm{y})$ of equation~\eqref{eq:variationalFE}.
This is impractical when solving~\eqref{eq:variationalFE} is expensive and many evaluations are required.
In such situations, it can be helpful to compute a surrogate $\tilde{g}_l$ of $g_l$ that can be used instead.
However, since creating $\tilde{g}_l$ typically requires multiple evaluations $g_l(\bm{y}_1), \ldots, g_l(\bm{y}_n)$ it is important that creating $\tilde{g}_l$ does not require more work than evaluating $g_l$ directly.
Fortunately, for many PDEs, we can take advantage of the solution's regularity using the multi-level method discussed in the next section.

\section{Work estimates for general multilevel methods}
\label{sec:sl_vs_ml}

Since evaluating $g_l$ typically becomes more costly as $l$ increases, it is natural to wonder if evaluations of lower-level $g_l$'s can help in computing a higher-level approximation $g_L$.
This is the principal idea behind multi-level methods, and the present section recalls how and when this structure can be utilised.

To quantify this statement, we suppose that there exists a quasi-normed space $\mathcal{S}\subseteq L^\infty(\mu)$ such that
$$
    g\in\mathcal{S}
    \qquad\text{and}\qquad g_l\in\mathcal{S}
$$
for all $l\in\mathbb{N}$ and that the approximations $g_l$ satisfy the following assumption.
\begin{enumerate}[left=1.5em, label={\textbf{(A\arabic{*})}}, ref={(A\arabic{*})}]
    \item\label{asn:level_error}
    There exist constants $A,\alpha> 0$ such that for all $l\in\mathbb{N}$,
    $$
        \|g - g_{l}\|_{\mathcal{S}} \le A 2^{-\alpha l}~.
    $$
\end{enumerate}

To approximate the QoIs $g_l$, we consider a family of model classes $\{\mathcal{M}_\gamma\}_{\gamma > 0}$ that can approximate any function $z\in\mathcal{S}$ with desired accuracy.
\begin{enumerate}[resume*]
    \item\label{asn:approx_error_eps}
    There exists a constant $C>0$ such that for all $\gamma > 0$ and $z\in\mathcal{S}$
    $$
        \inf_{v \in \mathcal{M}_\gamma} \|z - v\|_{L^\infty(\mu)} \le C \gamma \|z\|_{\mathcal{S}}~.
    $$
\end{enumerate}

To compute surrogates for the approximations $g_l$, we define the \emph{(weighted) least squares approximation}
\begin{equation} \label{eq:min_emp}
    P_{\mathcal{M}_\gamma}^n g_l \in \arg\min_{v\in\mathcal{M}_\gamma}\ \|g_l - v\|_{n}
    \quad\text{with}\quad
    \|v\|_{n} := \left(
        \frac{1}{n}\sum_{i=1}^n w(\bm y_i) |v(\bm y_i)|^2
    \right)^{1/2} \,.
\end{equation}
Here $w\ge 0$ may be any \emph{weight function} satisfying $\int w^{-1} \,\mathrm{d}\mu = 1$ and the
sample points $\bm y_i\sim w^{-1}\mu$ are i.i.d.\ for $i=1,\ldots,n$ and $n\in\mathbb{N}$.
We further assume that the sample size $n$ can be chosen such that the approximation $P_{\mathcal{M}_\gamma}^n g_l$ is almost optimal with any desired probability.
\begin{enumerate}[resume*]
    \item \label{asn:emp_approx_error}
    There exist $B,\beta > 0$ such that for all $p\in(0,1)$ and $\gamma>0$ there exist $n_0, n_{p,\gamma} \in\mathbb{N}$ with
    $$
        n_{p,\gamma} \le n_{0} \gamma^{-\beta}\log(p^{-1})
        \qquad\text{and}\qquad
        \|g - P_{\mathcal{M}_\gamma}^{n_{p,\gamma}}g\|_{L^2(\mu)}
        \le B \inf_{v\in\mathcal{M}_\gamma} \|g - v\|_{L^\infty(\mu)}
    $$
    with probability $1-p$ for every $g \in L^\infty(\mu)\subseteq L^2(\mu)$.
\end{enumerate}

To meaningfully discuss the work required for computing the surrogate $\tilde{g}_l$, we make one final assumption.
\begin{enumerate}[resume*]
    \item \label{asn:work_bound}
    There exist $\delta>\alpha\beta$ such that the computation time $\tau_l$ of $g_{l}(\bm{y})$ is independent of $\bm y$ and satisfies
    $$
        \tau_l \asymp 2^{\delta l}
    $$
    for all $\bm{y}\in U$.
\end{enumerate}
Although the computation time bound in assumption~\ref{asn:work_bound} can be relaxed, a bound of this form is essential to compare the overall cost of the single-level and multi-level approximations.
The condition $\delta > \alpha\beta$ ensures that the required work increases sufficiently with each level for the multi-level method to make sense.
If, for example, $\delta = 0$, the evaluation time would be the same for every level, and a multi-level estimate would only have the potential to increase computation time due to potential overhead costs.

\subsection{Single-level approximation}
Arguably, the easiest way to approximate $g$ is by using the single-level approximation
$$
    \tilde{g}_{L}^{\mathrm{SL}}
    := P_{\mathcal{M}_{\gamma_L}}^{n_L} g_L
$$
using the model class $\mathcal{M}_{\gamma_L}$ with $\gamma_L := 2^{-\alpha L}$ and $n_L := n_{p, \gamma_L} = n_{0} 2^{\alpha\beta L} \log(p^{-1})$ evaluations of $g_L$.
The parameter $p\in(0,1)$ can be chosen freely and controls the probability of success $1-p$ of the (random) least squares approximation method $P^{n_L}_{\mathcal{M}_{\gamma_L}}$.

To investigate the efficiency of this approach, we let $\tau^{\mathrm{SL}}_L$ denote the work required to compute $\tilde{g}_L^{\mathrm{SL}}$.
Since we, arguably, have to evaluate $g_L$ at least once, it makes sense to consider the work overhead
\begin{equation}
\label{eq:overhead}
    \omega^{\mathrm{SL}}_L := \frac{\tau^{\mathrm{SL}}_L}{\tau_L}
\end{equation}
as a measure of the efficiency of the approximation algorithm.
The single-level approximation satisfies the following error and work overhead bounds.

\begin{theorem}
\label{thm:sl_error_and_work}
    The single-level approximation satisfies the error bound
    $$
        \|g - \tilde{g}_{L}^{\mathrm{SL}}\|_{L^2(\mu)} 
        \lesssim 2^{-\alpha L}
    $$
    with probability $1-p$ and a work overhead bounded by $\omega^{\mathrm{SL}}_L \lesssim 2^{\alpha\beta L} \log(p^{-1})$.
\end{theorem}
\begin{proof}
    The proof proceeds similarly to the one presented in~\cite{TeckentrupJantschWebsterGunzburger2015}.
    Recall that $\tilde g^{\mathrm{SL}}_L$ is defined with $\gamma_L = 2^{-\alpha L}$ and $n_{L} = n_{0} 2^{\alpha\beta L} \log(p^{-1})$.
    Using the triangle inequality, assumption~\ref{asn:emp_approx_error}, assumption~\ref{asn:approx_error_eps} and assumption~\ref{asn:level_error}, the approximation error can thus be bounded with probability $1-p$ by
    \begin{align}
        \|g - \tilde{g}_{L}^{\mathrm{SL}}\|_{L^2(\mu)} 
        &\le \|g - g_{L}\|_{L^2(\mu)} + \|g_L - \tilde{g}_{L}^{\mathrm{SL}}\|_{L^2(\mu)} \\
        &\le \|g - g_{L}\|_{L^2(\mu)} + B \inf_{v\in\mathcal{M}_{\gamma_L}} \|g_{L} - v\|_{L^\infty(\mu)} \\
        &\le A\,2^{-\alpha L} + BC\gamma_L\|g_{L}\|_{\mathcal{S}} \\
        &\le A\,2^{-\alpha L} + BC\gamma_L\big(\|g\|_{\mathcal{S}} + A\,2^{-\alpha L}\big) \\
        &= (A + BC\|g\|_{\mathcal{S}} + ABC\,2^{-\alpha L}) \,2^{-\alpha L} .
    \end{align}

    The total work required to compute $\tilde g^{\mathrm{SL}}_L$ is bounded by $\tau^{\mathrm{SL}}_L = n_L \tau_L \le n_0 2^{\alpha\beta L} \log(p^{-1}) \tau_L$, resulting in the stated work overhead.
\end{proof}

\subsection{Multi-level approximation}

To utilise the regularity of the solution, we define the level-wise updates
$$
    \Delta g_l := g_l - g_{l-1}
    \qquad\text{with}\qquad
    g_0 := 0~.
$$
Given these, we then define the multi-level approximation as
$$
    \tilde{g}_L^{\mathrm{ML}} := \sum_{l=1}^L \Delta\tilde{g}_l^{\mathrm{ML}}
    \qquad\text{with}\qquad
    \Delta\tilde{g}_l^{\mathrm{ML}} := P_{\mathcal{M}_{\gamma_l}}^{n_l} \Delta g_l\,,
$$
where $\gamma_l := 2^{-\alpha(L-l)}$ and
\begin{equation}
\label{eq:level_wise_sample_size}
    n_l := n_{p_l,\gamma_l} = n_{0} 2^{\alpha\beta(L-l)} \log(p_l^{-1})
    \qquad\text{with}\qquad
    p_l := \frac{p}2 2^{-(L - l)} \,.
\end{equation}

Similar to the single-level case, we define the multi-level work $\tau^{\mathrm{ML}}_L$ and work overhead
\begin{equation}
\label{eq:ml_overhead}
    \omega^{\mathrm{ML}}_L := \frac{\tau^{\mathrm{ML}}_L}{\tau_L} .
\end{equation}
The multi-level approximation satisfies the following error and work overhead bounds.

\begin{theorem}
\label{thm:ml_error_and_work}
    The multi-level approximation satisfies the error bound
    $$
        \|g - \tilde{g}_{L}^{\mathrm{ML}}\|_{L^2(\mu)} 
        \lesssim 2^{-\alpha L} L
    $$
    with probability $1-p$ and a work overhead bounded by $\omega^{\mathrm{ML}}_L \lesssim \log(p^{-1})$.
\end{theorem}
\begin{proof}
    Using the triangle inequality and a telescoping sum, we can bound the multi-level approximation error by
    \begin{align}
        \|g - \tilde{g}_L^{\mathrm{ML}}\|_{L^2(\mu)}
        &\le \|g - g_L\|_{L^2(\mu)} + \|g_L -\tilde{g}_L^{\mathrm{ML}} \|_{L^2(\mu)}\\
        &\le A2^{-\alpha L} + \sum_{l=1}^L \|\Delta g_l -\Delta\tilde{g}_l^{ML} \|_{L^2(\mu)}~.
    \end{align}
    Now, recall that $\gamma_l = 2^{-\alpha(L-l)}$ and $n_{l} = n_{0} 2^{\alpha\beta(L-l)} \log(p_l^{-1})$ for $p_l := \frac{p}2 2^{-(L - l)}$.
    Using assumptions~\ref{asn:emp_approx_error} and~\ref{asn:approx_error_eps} each summand in the second term satisfies
    $$
        \|\Delta g_l - \Delta \tilde{g}_l^{\mathrm{ML}}\|_{L^2(\mu)}
        \le B \inf_{v\in\mathcal{M}_{\gamma_l}}\|\Delta g_l - v\|_{L^\infty(\mu)}
        \le BC\gamma_l \|\Delta g_l\|_{\mathcal{S}}
    $$
    with probability $1-p_l$.
    Using another triangle inequality and assumption~\ref{asn:level_error}, we further bound
    \begin{align}
        \|\Delta g_l - \Delta \tilde{g}_l^{\mathrm{ML}}\|_{L^2(\mu)}
        \le BC\gamma_l \|\Delta g_l\|_{\mathcal{S}}
        \le BC\gamma_l \big(\|g - g_l\|_{\mathcal{S}} + \|g - g_{l-1}\|_{\mathcal{S}}\big)
        \le ABC(1 + 2^\alpha)2^{-\alpha l}\gamma_l .
    \end{align}
    Using $\gamma_l = 2^{-\alpha (L-l)}$, this yields $\|\Delta g_l - \Delta \tilde{g}_l^{\mathrm{ML}}\|_{L^2(\mu)} \le ABC(1 + 2^\alpha)2^{-\alpha L}$.
    Reinserting this bound into the initial error bound for $\tilde g^{\mathrm{ML}}_L$ reveals that
    \begin{align}
        \|g - \tilde{g}_L^{\mathrm{ML}}\|_{L^2(\mu)}
        &\le A\,2^{-\alpha L} + ABC(1 + 2^\alpha) L\,2^{-\alpha L} \\
        &\le \Big(A + ABC(1 + 2^\alpha)\Big)L\, 2^{-\alpha L}
    \end{align}
    holds with probability $1-\tfrac{p}2\sum_{k=0}^{L-1} 2^{-k} \ge 1 - p$.
    The work required to compute $\tilde g^{\mathrm{ML}}_L$ is thus bounded by
    \begin{align}
        \tau^{\mathrm{ML}}_L
        &= \sum_{l=1}^L n_l\tau_l \\
        &= \sum_{l=1}^L n_0 2^{\alpha\beta(L-l)} \big(\log(2) + \log(p^{-1}) + (L-l)\log(2) \big) \tau_l \\
        &\lesssim 2^{\delta L} \log(p^{-1}) \sum_{l=1}^L (L-l+1) 2^{(\alpha\beta - \delta)(L-l)} \\
        &\le 2^{\delta L} \log(p^{-1}) \sum_{k=0}^\infty (k+1) 2^{(\alpha\beta - \delta)k} \\
        &\lesssim 2^{\delta L} \log(p^{-1}) \\
        &\lesssim \tau_L \log(p^{-1})~,
    \end{align}
    where convergence of the series in the second-to-last step follows from the ratio test, using $\alpha\beta-\delta < 0$.
\end{proof}

In the single-level setting, Theorem~\ref{thm:sl_error_and_work} shows that the upper bound for the work overhead $\omega^{\mathrm{SL}}_L\lesssim 2^{\alpha\beta L} \log(p^{-1})$ grows exponentially with the accuracy level $L$.
In contrast, in the multi-level setting, Theorem~\ref{thm:ml_error_and_work} shows that the work overhead $\omega^{\mathrm{ML}}_L \lesssim \log(p^{-1})$ is independent of the accuracy level $L$.
This is a clear advantage over the single-level method.

\begin{remark}
    We can remove the factor $L$ in the error bound of Theorem~\ref{thm:ml_error_and_work} by choosing $\gamma_l \lesssim 2^{-\alpha'(L-l)}$ for any $\alpha' > \alpha$.
    Since $\alpha < \tfrac\delta\beta$ by assumption~\ref{asn:work_bound}, we can choose any $\alpha' \in (\alpha, \tfrac\delta\beta)$ for this purpose.
    Note, however, that the pre-asymptotic constant depends significantly on the gap $\tfrac\delta\beta - \alpha$.
\end{remark}

%% file: content/theory.tex
\section{A multilevel method using sparse Tensor Train approximation}
\label{sec:application}

In this section, we detail our assumptions on the diffusion coefficient~\eqref{eq:diffusionCoefficientReal} and the approximation spaces $\{V_l\}_{l\in\mathbb{N}}$.
We define two families of model classes $\{\mathcal{M}^{1}_\gamma\}_{\gamma > 0}$ and $\{\mathcal{M}^2_\gamma\}_{\gamma > 0}$ and a joint regularity space $\mathcal{Q}$.
Then we show that assumptions~\ref{asn:level_error}--\ref{asn:work_bound} are satisfied with $\alpha < 1$, $\beta=p/(1-p)-\varepsilon$ (for any $\varepsilon>0$) and $\delta = d$ (the dimension of the physical domain $D$).
Finally, we discuss under which assumptions the constants that occur are finite independent of the expansion size $M$.

\subsection{General definitions}

Suppose that there exists a sequence $\bm{\rho}\in\mathbb{R}^M$ with $\bm{\rho} > 1$ such that the \emph{$\bm\rho$-weighted uniform ellipticity assumption}
\begin{equation}
\label{eq:rho-UEA}
\begin{aligned}
    r_{\bm{\rho}} &:= \inf_{x\in D} a_0(x) - \sum_{j\ge 1} \bm{\rho}_j \vert a_j(x) \vert  > 0 \\
    R_{\bm{\rho}} &:= \|a\|_{L^\infty(D\times B_{\mathbb{C}}(0, \bm\rho))} < \infty
\end{aligned}
\tag{$\bm{\rho}$\,-UEA}
\end{equation}
\noeqref{eq:rho-UEA}%
be satisfied.
This is a stronger version of the uniform ellipticity assumption and, indeed, equivalent to~\eqref{eq:UEA} with $c = r$ and $C=R$ for the constant sequence $\bm\rho = 1$.
We use~\eqref{eq:rho-UEA} to quantify the regularity of $\bm y \mapsto u(\bm y)$ and bound the convergence rate of polynomial approximations of this map.

Furthermore, we assume that the diffusion coefficient $a : D\times U \to \mathbb{R}$ is uniformly Lipschitz continuous in its first argument, i.e.,
\begin{equation}
\label{eq:lipschitz}
    \vert a(x_1, \bm{y}) - a(x_2, \bm{y}) \vert
    \le  \operatorname{Lip}(a)\,|x_1 - x_2|
\end{equation}
for some $\bm y$-independent constant $\operatorname{Lip}(a)>0$.

For every $\bm{y}\in U$ we can compute $u_l(\bm y)$ using the \emph{finite element method} (FEM).
For this purpose, we assume that $D\subset \mathbb{R}^d$ is a polyhydron and define a sequence of shape-regular, affine meshes $(\mathcal{T}_{l})_{l\in\mathbb{N}}$ covering $D$.
Following~\cite{Bouchot2017}, we denote by $h_l > 0$ the mesh width of $\mathcal{T}_{l}$ and assume that
\begin{equation}
\label{eq:meshWith}
    h_l = 2^{-l}h_0
\end{equation}
for some $h_0 \in (0, \mathrm{diam}(D))$.
We then choose $V_l \subset V$, as the $H^1$-conforming FE subspace of degree $\operatorname{deg}(V_l) \ge 1$ on the mesh $\mathcal{T}_{l}$, as defined in~\cite[Section~18.2]{Ern2021FEI}.

We represent the parameter-to-solution map $\bm y \mapsto u(\bm y)$ by a polynomial chaos expansion~\cite{Ernst2011}.
For this, let $\mu_1$ be the uniform probability measure on $[-1,1]$ and $\{L_k\}_{k\in\mathbb N}$ the $L^2(\mu_1)$-normalized univariate Legendre polynomials.
Since the normalized Legendre polynomials form an orthonormal basis of the Lebesgue space $L^2(\mu_1)$, their tensor products
$$
    L_{\bm{m}}(\bm{y}) := {\textstyle\prod_{k\in[M]}} L_{\bm{m}_k}(\bm{y}_k) \;.
$$
form an orthonormal basis of the product space $L^2(\mu) = L^2(\mu_1)^{\otimes M}$ with $\mu := \mu_1^{\otimes M}$.
Hence, every $v \in L^2(\mu; V) \simeq L^2(\mu)\otimes V$ has a unique representation of the form
\begin{equation}
\label{eq:legendre_expansion}
    v(\bm{y}) = \sum_{\bm{m} \in \mathbb{N}^M} \hat{\bm v}_{\bm{m}} L_{\bm{m}}(\bm{y})
    \quad \text{for all} \ \bm{y}\in U
\end{equation}
with coefficients
\begin{equation}
    \hat{\bm{v}}_{\bm{m}} = \int v(\bm{y})L_{\bm{m}}(\bm{y}) \, \mathrm{d}\mu(\bm{y}) \in V,
    \quad \bm{m}\in \mathbb{N}^M.
\end{equation}
Moving forward, we utilise the isometry $L^2(\mu)\otimes V \simeq \ell^2(\mathbb{N}^M)\otimes V$ to identify every element $v \in L^2(\mu)\otimes V$ with its uniquely determined coefficient tensor $\hat{\bm{v}} = (\hat{\bm{v}}_{\bm{m}})_{\bm{m}\in\mathbb{N}^M} \in \ell^{2}(\mathbb{N}^M) \otimes V$.

To simplify the presentation, we make extensive use of the multi-index notation and adopt the following conventions.
For all sequences $\bm v, \bm w\in\mathbb{R}^{I}$, indexed over a set $I$, we introduce the notation $\|\bm v\|_\bullet := \prod_{i\in I} |\bm v_i|$ and $\bm v^{\bm w} := \prod_{i\in I} \bm v_i^{\bm w_i}$.
Moreover, we often write operations and relations, like $\bm v \bm w$ or $\bm v \ge 1$, and understand them element-wise.

Using the Legendre series expansions~\eqref{eq:legendre_expansion}, we construct concrete model classes with structured coefficient sequences.
We do this in a similar fashion to~\cite{RauhutSchwab2017_CSPG,Bouchot2015}, where the theory is based on sparse Chebyshev polynomials.
Define the \emph{weight sequences} $\ubar{\bm{\omega}}, \bar{\bm\omega}\in [0, \infty]^{\mathbb{N}^M}$ via
\begin{align}
\label{eq:defOmega}
    \ubar{\bm{\omega}}_{\bm{\nu}} 
    &\coloneqq \|L_{\bm{\nu}}\|_{L^\infty(\mu)}
    = \|2\bm\nu + 1\|_{\bullet}^{1/2}
    \quad\text{and} \\
    \bar{\bm\omega}_{\bm\nu}
    &\coloneqq \|2\bm\nu + 1\|_{\bullet}^{-1/2} \bm\rho^{\nu}
\end{align}
for $\bm{\nu} \in \mathbb{N}^M$.
For the sake of simplicity, we use the trivial weight function $w=1$ in~\eqref{eq:min_emp}.
In general, we could define $\ubar{\bm\omega}$ in terms of the weighted norms $\|L_{\bm\nu}\|_{L^\infty_w(\mu)} := \|wL_{\bm\nu}\|_{L^\infty(\mu)}$, $\bm\nu\in\mathbb{N}^M$.
Given these weight sequences and any $p\in(0,1)$, we choose any weight sequence $\bm\omega\in [0, \infty]^{\mathbb{N}^M}$ satisfying
\begin{equation}
\label{eq:weight_sequence_constraints}
    \frac{\ubar{\bm\omega}^2}{\bm\omega} \in \ell^1
    \quad\text{and}\quad
    \frac{\bm\omega^{(2-p)/p}}{\bar{\bm\omega}}\in\ell^1 \,.
\end{equation}
Note that, since $\ubar{\bm\omega} \ge 1$, the condition $\frac{\ubar{\bm\omega}^2}{\bm\omega} \in \ell^1$ trivially implies $\frac{\ubar{\bm\omega}}{\bm\omega} \in \ell^\infty$.

Finally, for any weight sequence $\bm\omega \in [0,\infty]^{\mathbb{N}^M}$ and any $0<p\le\infty$ we define the weighted $\ell^p$-space~\cite{RauhutWard2016_WeightedL1,TrunschkeEigelNouy2025}
\begin{equation}
    \ell^p_{\bm\omega}
    := \left\{ \bm x \in \mathbb{R}^{\mathbb{N}^M}\ :\ \|\bm x\|_{\ell^p_{\bm\omega}}
    := \|\bm\omega \bm x\|_{\ell^p} < \infty\right\} \,.
\end{equation}
Also central to our presentation is the weighted $\ell_0$-``norm'' given by
\begin{equation}
    \|\bm x\|_{\ell^0_{\bm\omega}}
    = \sum_{\bm m\in\operatorname{supp}(\bm x)} \bm\omega^2_{\bm m},
\end{equation}
which counts the squared weights of the non-zero entries of $\bm x$.
When $\bm\omega \equiv 1$, these weighted norms reproduce the standard $\ell^p$ norms.

Given these defintions, we can define the regularity space
$$
    \mathcal{S}
    := \{z\in L^2(\mu) : \hat{\bm{z}}\in\ell^p_{\bm{\omega}^{(2-p)/p}}\} \,.
$$

\subsection{Definitions of the model classes}
\label{sec:model classes}

We start by recalling the tensor-network-based model classes introduced in~\cite{TrunschkeEigelNouy2025}.
Our model classes correspond to subsets of the coefficient tensors
$$
    \{\hat{\bm v} \in \ell^2(\mathbb{N}^M) : \hat{\bm{v}}_{\bm m} = 0\ \text{for all}\ \bm{m}\not\in [N]^M \}
$$
that can be efficiently encodes as sparse tensor networks.
A brief introduction to tensor networks can be found in appendix~\ref{app:tensors}.
We define the first model class
\begin{equation}
\label{eq:MRromega_sparse}
    \mathcal{M}^1_{R,r,\bm\omega}
    := \bigcap_{m\in[M]} \bigcup_{Q\in\mathcal{Q}^1_{m,R,\bm\omega}} Q\mathcal{C}^1_{Q,r,\bm\omega}~,
\end{equation}
with
\begin{alignat*}{2}
    \mathcal{Q}^1_{m,R,\bm\omega} :=\ 
    &\{\,
    \operatorname{unfold}_{m}\big(U^{(1)}\circ\cdots\circ U^{(m-1)}\big)
    \otimes I_{d_m} \otimes
    \operatorname{unfold}_{1}\big(V^{(m+1)}\circ\cdots\circ V^{(M)}\big)^\intercal\\
    &\!: U^{(j)}, V^{(j)}\in\{0,1\}^{r_{j-1}\times d_j\times r_j}\text{ with }r_{j-1},r_j \le R \\
    &\!: \text{all } U^{(j)} \text{ are left-orthogonal and  $r_j$-sparse} \\
    &\!: \text{all } V^{(j)} \text{ are right-orthogonal and $r_{j-1}$-sparse}
    \,\}
\end{alignat*}
and
\begin{equation}
\label{eq:model_class_1_core}
    \mathcal{C}^1_{Q,r,\bm\omega}
    := B_{\ell^0_{Q^\intercal\bm\omega}}(0, r)
    = \big\{C \in\mathbb{R}^{r_{m-1}\times d_m\times r_m} \,:\, \|C\|_{\ell^0_{Q^\intercal \bm\omega}}\le r\big\}~.
\end{equation}

For the choice of $N\in \mathbb{N}$, consider $QC \in \mathcal{M}^1_{R,r,\bm\omega}$. 
By definition of $\mathcal{C}^1_{Q,r,\bm\omega}$, it holds $\|C\|_{\ell^0_{Q^\intercal\bm\omega}} \le r$, hence $$\sum_{\bm m \in \supp(QC)} \bm\omega_{\bm m}^2 \le r.$$ In particular, $\bm\omega_{\bm m}^2 \le r$ for all $\bm m \in \supp(QC)$.
Assume $\|\frac{\underline{\bm\omega}}{\bm\omega}\|_{\ell^\infty} \le 1$ with $\underline{\bm\omega}$ given by \eqref{eq:defOmega}.
Then $\|2\bm m+1\|_\bullet \le r$ on $\supp(QC)$. Choosing $$N \ge (r-1)/2$$ yields $\supp(QC) \subset [N]^M$, so that $\mathcal{M}^1_{R,r,\bm\omega}$ is contained in the
coefficient space supported on $[N]^M$.\\

To define the second model class, we require the following definition.
\begin{definition}
    We say that a matrix $Q$ is \emph{$\bm{\omega}$-orthogonal} if $Q^\intercal\operatorname{diag}(\bm{\omega})Q$ is diagonal.
\end{definition}

With this in hand, the second model class can be written as
\begin{equation}
\label{eq:MRromega}
    \mathcal{M}^{2}_{R,r,\bm{\omega}}
    := \bigcap_{m\in[M]} \bigcup_{Q\in\mathcal{Q}^2_{m,R,\bm{\omega}}} Q\mathcal{C}^2_{Q,r,\bm{\omega}}~,
\end{equation}
with
\begin{alignat*}{2}
    Q\in\mathcal{Q}^2_{m,R,\bm\omega} :=\ 
    &\{\,
    Q\in\mathcal{L}\big(\mathbb{R}^{r_{m}\times d_m\times r_{m+1}}, \mathbb{R}^{d_1\times \cdots\times d_M}\big)
    \\
    &\!: \text{$Q$ is orthogonal and $\boldsymbol{\omega}^{2}$-orthogonal and $r_m, r_{m+1} \le R$}
    \,\}
\end{alignat*}
and
\begin{equation}
    \mathcal{C}^2_{Q,r,\bm{\omega}}
    := \big\{C \in\mathbb{R}^{r_{m-1}\times d_m\times r_m} \ \big|\  \|C\|_{\ell^0_{\bm{\omega}_Q}}\le r \text{ with } \bm{\omega}_Q^2 := \operatorname{diag}(Q^\intercal\operatorname{diag}(\bm{\omega}^{2})Q)\big\}~.
\end{equation}

In contrast to the model class $\mathcal{M}^1_{R,r,\bm\omega}$, the restriction to coefficients supported on $[N]^M$ can in principle
lead to a smaller set than $\mathcal{M}^2_{R,r,\bm\omega}$. However, since the error bounds for $\mathcal{M}^2_{R,r,\bm\omega}$ rely on its closeness to other sparse spaces,\footnote{Specifically, $\mathcal{M}^1_{R,r,\bm\omega} \subseteq \mathcal{M}^2_{R,r,\bm\omega}$ and $d_{\mathrm{si}L^\infty_w}(\mathcal{M}^2_{R,r,\bm\omega}, B_{\ell^0_{\tilde{\bm\omega}}}(0, (1+\tfrac1{\varepsilon})^2r) \le \varepsilon$ for certain $\tilde{\bm\omega}$, see Theorem~\ref{thm:RIP_MRromega}.} we argue that similar error bounds can be obtained for a suitable choice of $N$.\\

These two classes give rise to the two families $\{\mathcal{M}^{1/2}_{\gamma}\}_{\gamma>0}$ defined by
\begin{equation}
    \mathcal{M}^{1/2}_\gamma
    := \{
        v\in Q : \hat{\bm{v}} \in \mathcal{M}^{1/2}_{\infty, r(\gamma), \bm\omega}
    \}
    \,,\quad \gamma>0\,,
\end{equation}
with $r(\gamma) > 0$ defined later in~\eqref{eq:r_gamma}.
We use the rank constraint $R=\infty$ since we are using the rank-adaptive algorithms from~\cite{TrunschkeEigelNouy2025} during the approximation.

\input{content/asn_proofs/level_error}

\input{content/asn_proofs/approx_error_eps}

\input{content/asn_proofs/emp_approx_error}
\input{content/asn_proofs/work_bound}
\input{content/asn_proofs/verification}

%% file: content/asn_proofs/level_error.tex
\subsection{Proof of assumption~\ref{asn:level_error}}

We start by proving the regularity assumption $g\in \mathcal{S}$ and $g_l\in \mathcal{S}$ for all $l\in\mathbb{N}$.
Then, we continue to prove the desired approximation error bound.
To do this, we define for $\bm{\rho}\in(0, \infty)^{M}$ the open polydiscs $B_{\mathbb{C}}(0, \bm{\rho}) := \prod_{m\in[M]} B_{\mathbb{C}}(0, \bm{\rho}_m)$.

\subsubsection*{Regularity assumption}

Let $\hat{\bm g}$ and $\hat{\bm g}_l$ denote the Legendre series coefficients of $g$ and $g_l$, respectively.
Then $g\in \mathcal{S}$ and $g_l\in \mathcal{S}$ for all $l\in\mathbb{N}$ is equivalent to $\hat{\bm g}\in\ell^p_{\bm\omega^{(2-p)/p}}$ and $\hat{\bm g}_l\in\ell^p_{\bm\omega^{(2-p)/p}}$ for all $l\in\mathbb{N}$.
These summability conditions can be shown with help of the subsequent two theorems.

\begin{theorem}
\label{thm:legendre_coefficient_bounds}
    Let $\bm\rho\in(1,\infty)^M$, $X$ be a Hilbert space and $v : B_{\mathbb{C}}(0, \bm\rho) \to X$ be holomorphic such that 
    $\|v\|_{L^\infty(B_{\mathbb{C}}(0, \bm\rho); X)} < \infty$.
    Moreover, let $p\in(0,\infty]$, and $\tilde{\bm\omega}\in(0,\infty)^{\mathbb{N}^M}$.
    Then the Legendre series coefficients $\hat{\bm{v}}$ of $v$ satisfy
    $$
        \|\hat{\bm{v}}\|_{\ell^p_{\tilde{\bm\omega}} \otimes X}
        \le \big\|\tfrac{\bm\rho^2}{\bm\rho^2 - 1}\big\|_{\bullet} \|\tfrac{\tilde{\bm\omega}}{\bar{\bm\omega}}\|_{\ell^p}  \|v\|_{L^\infty(B_{\mathbb{C}}(0, \bm\rho); X)}
    $$
    with $\bar{\bm\omega}_{\bm m} \coloneqq \|2\bm{m}+1\|_{\bullet}^{-1/2} \bm{\rho}^{\bm{m}}$.
    In particular, $\hat{\bm v}\in\ell^2\otimes X$ and $\hat{\bm v}\in\ell_{\bm\omega^{(2-p)/p}}^p\otimes X$ for all $p > 0$.
\end{theorem}
This theorem is a slight generalisation of Theorem~3.4 in~\cite{TrunschkeEigelNouy2025} and its proof can be found in Appendix~\ref{app:proof:thm:legendre_coefficient_bounds}.

\begin{theorem}
\label{thm:supremum_norm_bound}
    Let $u$ be the solution of the diffusion equation~\eqref{eq:darcy} and $u_l$ the corresponding Galerkin projection onto $V_l$.
    Moreover, assume that the affine coefficients~\eqref{eq:diffusionCoefficientReal} satisfy~\eqref{eq:rho-UEA}.
    Then $u : B_{\mathbb{C}}(0, \bm\rho) \to V$ and $u_l : B_{\mathbb{C}}(0, \bm\rho) \to V$ are holomorphic with
    $$
        \|u\|_{L^\infty(B_{\mathbb{C}}(0, \bm\rho); V)}
        \le \tfrac{1}{c} \|f\|_{V^*}
        \,,
        \quad 
        \|u_l\|_{L^\infty(B_{\mathbb{C}}(0, \bm\rho); V)}
        \le \tfrac{1}{c} \|f\|_{V^*}
        \,.
    $$
\end{theorem}
This theorem is a slight generalisation of Theorem~3.2 in~\cite{TrunschkeEigelNouy2025}.
The proof is given in Appendix~\ref{app:proof:thm:supremum_norm_bound}.

To prove the inclusions $\hat{\bm g}\in\ell^p_{\bm\omega^{(2-p)/p}}$ and $\hat{\bm g}_l\in\ell^p_{\bm\omega^{(2-p)/p}}$, note that Theorem~\ref{thm:supremum_norm_bound} ensures that $u$ and $u_l$ are analytic and uniformly bounded on $B_{\mathbb{C}}(0, \bm\rho)$.
Therefore, $g = G\circ u$ and $g_l = G \circ u_l$ are analytic and their norm is uniformly bounded by
\begin{align}
    \|g\|_{L^\infty(B_{\mathbb{C}}(0, \bm\rho))}
    &\le \|G\|_{V^*} \|u\|_{L^\infty(B_{\mathbb{C}}(0, \bm\rho); V)}
    \le \|G\|_{V^*} \tfrac{1}{c}\|f\|_{V^*}\quad\text{and}\\
    \|g_l\|_{L^\infty(B_{\mathbb{C}}(0, \bm\rho))}
    &\le \|G\|_{V^*} \|u\|_{L^\infty(B_{\mathbb{C}}(0, \bm\rho); V)}
    \le \|G\|_{V^*} \tfrac{1}{c}\|f\|_{V^*}
    \,.
\end{align}
Theorem~\ref{thm:legendre_coefficient_bounds} thus ensures the desired summability properties.

\subsubsection*{Approximation error bound}

Our proof of the approximation error bound relies on the following approximation property of the discretisation spaces $V_l$.
\begin{theorem}[Theorem~32.2~in~\cite{Ern2021FEII}.]
\label{thm:smoothness_scales}
    There is a constant $C'> 0$ such that for $0 < t \le \operatorname{deg}(V_l)$ and $l\in\mathbb{N}_0$
    \begin{equation}
        \inf_{w \in V_l}\Vert v - w \Vert_{H_0^1(D)} 
        \le C' h_l^t \Vert v \Vert_{H^{1+t}(D)}
        \quad\text{for all}\quad
        v \in H^{1+t}(D)
        \,.
    \end{equation}
\end{theorem}

Moreover, retracing the proof of Theorem~2 in~\cite{Savare1998} it can be shown that $u(\bm{y})$ is an element of the Sobolev--Slobodeckij space $H^{1+t}(D)$.
\begin{theorem}
\label{thm:guaranteeSolutionSmooth}
    Suppose that $a$ satisfies~\eqref{eq:UEA} and~\eqref{eq:lipschitz}.
    For a given $f\in V^*$, let $u(\bm{y})$ denote the solution to~\eqref{eq:variational_form}.
    There exists a constant $C'' > 0$, independent of $f$, such that for every $\bm{y}\in U$
    \begin{equation}
        \|u(\bm{y})\|_{H^{1+t}} \le C'' \|f\|_{H^{-1+t}}
        \quad\text{for all}\quad
        t \in [0, 1/2)\;.
    \end{equation}
\end{theorem}

To prove assumption~\ref{asn:level_error}, we first use
the boundedness of $G$ together with Céa's lemma to bound
\begin{equation}
    |g(\bm y) - g_l(\bm y)|
    \le \|G\|_{V^*} \| u(\bm{y}) - u_\ell(\bm{y}) \|_V
    \le \|G\|_{V^*} \big(\tfrac{R}{r}\big)^{1/2}
    \inf_{v_\ell \in V_\ell} \| u(\bm{y})- v_\ell \|_V
    \,,
\end{equation}
where $r$ and $R$ are the bounds from~\eqref{eq:UEA}.
Applying Theorems~\ref{thm:smoothness_scales} and~\ref{thm:guaranteeSolutionSmooth} with $h_l$ as defined in~\eqref{eq:meshWith} then yields
\begin{align}
    |g(\bm y) - g_l(\bm y)|
    \lesssim \inf_{v_\ell \in V_\ell} \| u(\bm{y})- v_\ell \|_V
    \lesssim 2^{-lt} \| u(\bm{y}) \|_{H^{1+t}(D)}
    \lesssim 2^{-lt} \| f \|_{H^{-1+t}(D)}
    \,.
\end{align}
Since $g$ and $g_l$ are analytic from $B_{\mathbb{C}}(0,\bm\rho)$ to $\mathbb{R}$, their difference is analytic too, and Theorem~\ref{thm:legendre_coefficient_bounds} ensures
$$
    \|g - g_l\|_{\mathcal{S}}
    = \|\hat{\bm g} - \hat{\bm g}_l\|_{\ell^p_{\bm\omega^{(2-p)/p}}}
    \lesssim \|g - g_l\|_{L^\infty(B_{\mathbb{C}}(0, \bm\rho))}
    \lesssim 2^{-lt} \| f \|_{H^{-1+t}(D)}
    \,.
$$

%% file: content/asn_proofs/approx_error_eps.tex
\subsection{Proof of assumption~\ref{asn:approx_error_eps}}

Without rank restrictions, the model class $\mathcal{M}^1_{\infty, r,\bm\omega}$ is equal to the set of weighted sparse vectors
$$
    \mathcal{M}^1_{\infty,r,\bm\omega} = B_{\ell^0_{\bm\omega}}(0, r(\gamma)) \,.
$$
Moreover, note that the matrices $Q\in\mathcal{Q}^1_{m,\infty,\bm\omega}$ are quasi-permutation matrices (every column is a standard basis vector) and thereby, trivially, orthogonal and $\bm\omega$-orthogonal.
This implies, that
$$
    \mathcal{M}^1_\gamma = \mathcal{M}^1_{\infty,r(\gamma),\bm\omega} \subseteq \mathcal{M}^2_{\infty,r(\gamma),\bm\omega} = \mathcal{M}^2_\gamma \,.
$$

To prove assumption~\ref{asn:approx_error_eps}, note that for all $z\in\mathcal{S}$ and $v\in\mathcal{M}^{1/2}_{\infty,r,\bm\omega}$
$$
    \|z - v\|_{L^\infty(\mu)}
    \le \|\hat{\bm z} - \hat{\bm v}\|_{\ell^1_{\ubar{\bm{\omega}}}}
    \,.
$$
Using this, we can bound
\begin{align}
    \inf_{v\in\mathcal{M}^2_\gamma} \|z - v\|_{L^\infty(\mu)}
    &\le \inf_{v\in\mathcal{M}^1_\gamma} \|z - v\|_{L^\infty(\mu)} \\
    &= \inf_{v\in B_{\ell^0_{\bm\omega}}(0, r(\gamma))} \|z - v\|_{L^\infty(\mu)} \\
    &\le \inf_{\|\hat{\bm v}\|_{\ell^0_{\bm\omega}} \le r(\gamma)} \|\hat{\bm z} - \hat{\bm v}\|_{\ell^1_{\ubar{\bm\omega}}} \\
    &\le \|\tfrac{\ubar{\bm\omega}}{\bm\omega}\|_{\ell^\infty} \inf_{\|\hat{\bm v}\|_{\ell^0_{\bm\omega}} \le r(\gamma)} \|\hat{\bm z} - \hat{\bm v}\|_{\ell^1_{\bm\omega}} \\
    &\le r(\gamma)^{-s} \|\hat{\bm q}\|_{\ell^p_{\omega^{(2-p)/p}}}
\end{align}
with $s = \frac1p - 1$.
Here, the second-to-last inequality follows from Lemma~2.11 in~\cite{TrunschkeEigelNouy2025}.
The upper bound remains finite, since $\|\frac{\ubar{\bm\omega}}{\bm\omega}\|_{\ell^\infty} \le \|\frac{\ubar{\bm\omega}}{\bm\omega}\|_{\ell^1} < \infty$.\footnote{Under reasonable assumptions, we even have $\|\frac{\ubar{\bm\omega}}{\bm\omega}\|_{\ell^\infty}\le 1$.}
The remaining infimum is precisely the \emph{$\bm\omega$-weighted $r(\gamma)$-sparse approximation error} proposed in~\cite{RauhutWard2016_WeightedL1}.
The last inequality bounds this infimum using Corollary~2.8 in~\cite{TrunschkeEigelNouy2025}.
Choosing
\begin{equation}
\label{eq:r_gamma}
    r(\gamma) := \gamma^{-1/s}
    \,,\quad
    s = \frac1p - 1
\end{equation}
yields the desired bound
$$
    \inf_{v\in\mathcal{M}^2_\gamma} \|q - v\|_{L^\infty(\mu)}
    \le \inf_{v\in\mathcal{M}^1_\gamma} \|q - v\|_{L^\infty(\mu)}
    \le C \gamma \|\hat{\bm q}\|_{\ell^p_{\omega^{(2-p)/p}}}
    = C \gamma \|q\|_{\mathcal{S}}
$$
with $C = \|\frac{\ubar{\bm\omega}}{\bm\omega}\|_{\ell^\infty}$.

%% file: content/asn_proofs/emp_approx_error.tex
\subsection{Proof of assumption~\ref{asn:emp_approx_error}}

The proof of this assumption proceeds in two steps.
First we show that the approximation $P_{\mathcal{\gamma}}^n g_l$ exhibits the desired error bounds, assuming that the \emph{restricted isometry property} (RIP)~\cite{Candes2008,Adcock2022,Trunschke2021} is satisfied.
Then, we show that this property is satisfied with high probability.

We begin by recalling the definition of the RIP.
For a given set of functions $A$, we define
\begin{equation}
    \operatorname{RIP}_{\mathcal{A}}(\delta)
    \quad\Leftrightarrow\quad
    (1-\delta)\|u\|_{L^2(\mu)}^2 \le \|u\|_{n}^2 \le (1+\delta)\|u\|^2
    \quad\text{for all}\quad
    u\in A \,.
\end{equation}
If this property is satisfied for a parameter $\delta\in(0,1)$, the error of estimator~\eqref{eq:min_emp} can be bounded as follows.
\begin{theorem}[Theorem~5.3~in~\cite{Adcock2022}]
\label{thm:rip_error}
    Suppose that $\operatorname{RIP}_{\mathcal{A}}(\delta)$ holds for the Minkowski (element-wise) difference $\mathcal{A} = \{P^n_{\mathcal{M}_\gamma}g_l\}-\mathcal{M}_\gamma$, then
    \begin{align}
        \Vert g_l - P_{\mathcal{M}_\gamma}^n g_l \Vert_{L^2(\mu)}
        &\le 
        \left( 1 + \frac{2}{\sqrt{1-\delta}} \right)
        \inf_{h \in \mathcal{M}_{\gamma}} 
        \Vert g - h \Vert_{L^\infty(\mu)}.
    \end{align}
\end{theorem}

We now continue to prove that the RIP holds for both model classes with high probability.
For $\mathcal{M}^1_{\infty,r,\bm\omega}$, this is proven by a straight-forward generalisation of Corollary~4.12 in~\cite{TrunschkeEigelNouy2025}.
\begin{theorem}
\label{thm:sparse_sample_complexity}
	Fix parameters $\delta,p\in(0,1)$ and $\gamma > 0$ and let $r(\gamma)$ be defined as in~\eqref{eq:r_gamma}.
    Then there exists a universal constant $c > 0$ such that at most
	\begin{equation}
		n = c \delta^{-2} r(\gamma)\max\{M\log^3(r(\gamma)) \log(N), \log(p^{-1})\}
	\end{equation}
    i.i.d.\ sample points $\bm y_1,\ldots, \bm y_n\sim w^{-1}\mu$ are sufficient to ensure $\operatorname{RIP}_{\mathcal{A}}(\delta)$ for $\mathcal{A} = \{P_{\mathcal{M}^1_{\gamma}}^n g_l\} - \mathcal{M}^1_{\gamma}$ with a probability of at least $1-p$.
\end{theorem}
\begin{proof}
    Note that $\mathcal{M}^1_{\gamma} \simeq \mathcal{M}^1_{\infty, r(\gamma),\bm\omega}$ and, by Proposition~4.11 in~\cite{TrunschkeEigelNouy2025}, it holds that
    $$
        \mathcal{A}
        = \{P_{\mathcal{M}^1_{\gamma}}^n g_l\} - \mathcal{M}^1_{\gamma}
        \subseteq \mathcal{M}^1_{\gamma} - \mathcal{M}^1_{\gamma}
        = \mathcal{M}^1_{\infty,2r(\gamma),\bm\omega}
        =: \mathcal{B} \,.
    $$
    Since $\operatorname{RIP}_{\mathcal{B}}(\delta)$ implies $\operatorname{RIP}_{\mathcal{A}}(\delta)$, it suffices to prove $\operatorname{RIP}_{\mathcal{B}}(\delta)$.
	Recall that $\{L_{\bm m}\}_{\bm m\in\mathbb{N}^M}$ is an orthonormal basis for $L^2(\mu)$ and that by assumption~\eqref{eq:weight_sequence_constraints} there exists a constant $c_1$ such that $c_1 \bm\omega \ge \ubar{\bm\omega}$ with $\ubar{\bm\omega}_{\bm m} = \|w^{1/2} L_{\bm m}\|_{L^\infty(\mu)}$.
    Moreover, note that
    $$
        \mathcal{B}
        = \mathcal{M}^1_{\infty,2r(\gamma),\bm\omega}
        = \mathcal{M}^1_{\infty,2c_1r(\gamma), c_1\bm\omega}
        =: \tilde{\mathcal{B}} \,.
    $$
    Considering $\tilde{\mathcal{B}}$, the conditions for Corollary~4.12 in~\cite{TrunschkeEigelNouy2025} are satisfied and there exists a constant $c_2$ such that using
    $$
        n \le c_2 \delta^{-2} (2c_1r(\gamma))\max\{M \log^3(2c_1r(\gamma))\log(N), \log(p^{-1})\}
    $$
    i.i.d.\ sample points $\bm y_1,\ldots, \bm y_n\sim w^{-1}\mu$ ensures $\operatorname{RIP}_{\tilde{\mathcal{B}}}(\delta)$ with a probability of at least $1-p$.
\end{proof}

Combining Theorems~\ref{thm:rip_error} and~\ref{thm:sparse_sample_complexity} shows for every $g \in L^\infty(\mu)\subseteq L^2(\mu)$ that
$$
    \|g - P_{\mathcal{M}^1_\gamma}^{n_{p,\gamma}}g\|_{L^2(\mu)}
    \le \underbrace{\left( 1 + \frac{2}{\sqrt{1-\delta}} \right)}_{=: B} \inf_{v\in\mathcal{M}^1_\gamma} \|g - v\|_{L^\infty(\mu)}
$$
with probability $1-p$ when
\begin{align}
    n_{p,\gamma}
    &\ge c \delta^{-2} r(\gamma)\max\{M\log^3(r(\gamma)) \log(N), \log(p^{-1})\} \\
    &\ge \tilde{c} c \delta^{-2} r(\gamma)^{1+\varepsilon} \log(p^{-1}) \\
    &= \tilde{c} c \delta^{-2} \gamma^{-1/s+\varepsilon} \log(p^{-1}) \,. \label{eq:npgamma}
\end{align}
Note that the last inequality ignores the logarithmic terms by increasing the multiplicative constant by a factor of $\tilde{c}$ and slightly worsening the asymptotic rate by $\varepsilon$.
Defining $n_{0} := \tilde{c}c\delta^{-2}$ and $\beta := 1/s -\varepsilon= p/(1-p) -\varepsilon$ proves assumption~\ref{asn:emp_approx_error} for $\mathcal{M}^1_{\gamma}$.

For $\mathcal{M}^2_{\infty, r, \bm\omega}$ the RIP is proven by a simple reformulation of Corollary~5.14 in~\cite{TrunschkeEigelNouy2025}.
\begin{theorem}
\label{thm:semi-sparse_sample_complexity}
    Fix parameters $\delta,p\in(0,1)$ and $\gamma > 0$, and let $r(\gamma)$ be defined as in~\eqref{eq:r_gamma}.
    Then there exists a universal constant $C>0$ such that at most
    \begin{equation}
        n = C \delta^{-2} r(\gamma)\max\big\{M\log^3(r(\gamma))\log(N), \log(p^{-1})\big\}
    \end{equation}
    i.i.d.\ sample points $\bm y_1,\ldots, \bm y_n\sim w^{-1}\mu$ are sufficient to ensure $\operatorname{RIP}_{\mathcal{A}}(\delta)$ for $\mathcal{A} = \{P_{\mathcal{M}^2_{\gamma}}^n g_l\} - \mathcal{M}^2_{\gamma}$ with a probability of at least $1-p$.
\end{theorem}
The proof of this theorem relies on a slight generalisation of Theorem~5.12 in~\cite{TrunschkeEigelNouy2025} and can be found in Appendix~\ref{app:proof:thm:semi-sparse_sample_complexity}.

Combining Theorems~\ref{thm:rip_error} and~\ref{thm:semi-sparse_sample_complexity} proves that assumption~\ref{asn:emp_approx_error} is satisfied for $\mathcal{M}^2_\gamma$ with $B = (1 + \frac{2}{\sqrt{1-\delta}})$, $n_{1,1} := \tilde{c}c\delta^{-2}$ and $\beta := 1/s -\varepsilon= p/(1-p) -\varepsilon$.

\begin{remark}
    The presence of $M$ in the sample size bounds of Theorems~\ref{thm:sparse_sample_complexity} and~\ref{thm:semi-sparse_sample_complexity} stems from the fact that, for general weight sequences $\bm{\omega}$, the sparsity constraint $\|C\|_{\ell^0_{\bm\omega}} \le r$ does not restrict the number of ``active variables'' of the functions in our model class.
    Even though we investigate approximation of functions with a potentially unbounded number of variables, this is not an issue in our setting, since we can choose anisotropic weight sequences like $\bm\omega_{\bm{\nu}} = \prod_{m=1}^M m^{\beta \bm{\nu}_m}$ (cf.~Section~\ref{sec:experiments}) to effectively truncate to $M<\infty$.
\end{remark}

%% file: content/asn_proofs/work_bound.tex
\subsection{Proof of assumption~\ref{asn:work_bound}}

We assume that the solver for the Galerkin discretized equation of level $\ell$ has a computational complexity scaling linearly with the dimension of $V_\ell$.
This can be achieved e.g.\ by multigrid solvers.
This means, that the work required for computing $g_l(\bm y)$ scales like
\begin{equation}
\label{eq:work_scale_Galerkin_dim}
    \tau_l \asymp \dim(V_l) \,.
\end{equation}
To bound this dimension, recall that the degrees of freedom of finite element methods scale like
$$
    h^{-d}q^2
$$
with respect to the mesh-width $h$, the physical domain dimension $d$ and the FE order $q$.
Since we assume that $h_l\asymp 2^{-l}$, we can conclude
\begin{equation}
    \label{eq:Galerkin_dim_scale_spacial_dim}
    \dim(V_l)
    \asymp 2^{dl} \,.
\end{equation}
Assumption~\ref{asn:work_bound} is therefore satisfied with $\delta = d$.

%% file: content/asn_proofs/verification.tex
\subsection{\texorpdfstring{$M$}{M}-independent boundedness of all constants}
\label{sec:verification}

In this section, we consider a diffusion coefficient of the form
\begin{equation}
    \label{eq:example_diffcoeff}
    a(x, \bm{y}) = a_{0} + \sum_{m=1}^{M} \bm{y}_{m}a_{m}(x), \quad \text{with} \quad a_{m}(x) = m^{-\eta}\prod_{k=1}^d\cos(\pi mx_k) ,
\end{equation}
Moreover, we let 
\begin{equation}
\label{eq:omega_xi}
    \bm\rho_m := cm^\theta
    \quad\text{and}\quad \bm\omega := (\bm\xi\bar{\bm\omega})^{p/(2-p)}
    \quad\text{with}\quad \bm\xi_{\bm\nu} := \bm\rho^{-\kappa\bm\nu}
\end{equation}
and compute the set of admissible parameters $\eta\in(1,\infty)$, $c>0$, $\theta\in[0,\infty)$, $\kappa\in[0,\infty)$, and $p\in(0,1)$ for which the stationary diffusion PDE~\eqref{eq:darcy} with coefficient~\eqref{eq:example_diffcoeff} satisfies the theoretical assumptions~\ref{asn:level_error}--\ref{asn:work_bound}.
For the convenience of the reader, we recall the four conditions requiring verification:
\begin{enumerate}[label=(\roman*)]
    \item The sequence $\bm\rho$ satisfies $\sum_{m\ge 1}\bm\rho_{m}\|a_{m}\|_{L^{\infty}} \le a_{0}-\delta$ for some $\delta > 0$, implying~\eqref{eq:rho-UEA}
    \label{itm:rho-UEA}
    \item The upper bound in Theorem~\ref{thm:legendre_coefficient_bounds} is finite, i.e., $\big\|\frac{\bm\rho^{2}}{\bm\rho^{2}-1}\big\|_{\bullet} < \infty$.
    \label{itm:rho_product_finite}
    \item The weight sequence satisfies the second inequality in~\eqref{eq:weight_sequence_constraints}, i.e.\ $\frac{\bm\omega^{(2-p)/p}}{\bar{\bm\omega}} \in \ell^{p}$ for $p \in (0,1)$.
    \label{itm:omega_lp}
    \item The weight sequence satisfies the first inequality in~\eqref{eq:weight_sequence_constraints}, i.e.\ $\frac{\ubar{\bm\omega}^{2}}{\bm\omega} \in \ell^{1}$.
    \label{itm:inv_omega_l1}
\end{enumerate}

To verify conditions \ref{itm:rho_product_finite}--\ref{itm:inv_omega_l1}, which involve potentially infinite products, we use the following standard bound.

\begin{lemma}
\label{lem:product_bound}
    Let $\pi \in [0,\infty)^{\mathbb{N}}$.
    Then,
    $1 + \|\pi\|_{\ell^{1}}
        \le \|1 + \pi\|_{\bullet}
        \le \exp(\|\pi\|_{\ell^{1}})$.
\end{lemma}
\begin{proof}
    Expanding the finite product $\prod_{m=1}^{M}(\pi_{m}+1)$ and collecting terms by monomial degree yields
    \begin{align}
        \prod_{m=1}^{M}(\pi_{m} + 1)
        &=\prod_{m=1}^{M}(|\pi_{m}| + 1) \\
        &= 1 + \sum_{m=1}^{M}|\pi_{m}| + [\text{higher degree terms}] \\
        &\ge 1 + \sum_{m=1}^{M} |\pi_{m}| \,.
    \end{align}
    Taking the limit $M\to\infty$ proves the lower bound.
    For the upper bound, we use the inequality $1+x \le \exp(x)$ for $x \ge 0$ to conclude that
    $$
        \prod_{m=1}^{M}(\pi_{m}+1)
        \le \prod_{m=1}^{M}\exp(\pi_m)
        = \exp\left(\sum_{m=1}^{M}|\pi_{m}|\right).
    $$
    Taking the limit $M \to \infty$ proves the claim.
\end{proof}

\subsubsection*{Condition~\ref{itm:rho-UEA}.}
Substituting $\bm\rho_m = cm^\theta$ and noting that $\|a_m\|_{L^\infty(D)} = m^{-\eta}$, the admissibility condition becomes
$$
    \sum_{m=1}^{\infty} c\, m^{\theta} m^{-\eta}
    = c \zeta(\eta-\theta)
    \le a_{0} - \delta .
$$
This condition is satisfied for any $\theta < \eta - 1$ and $\delta \in (0, a_0)$, provided we choose $c \le \frac{a_{0}-\delta}{\zeta(\eta-\theta)}$.

\subsubsection*{Condition~\ref{itm:rho_product_finite}.}
Under the condition that $c > 1$ we can use Lemma~\ref{lem:product_bound} to bound
$$
    \left\|\frac{\bm\rho^{2}}{\bm\rho^{2}-1}\right\|_{\bullet}
    = \left\| 1 + \frac{1}{\bm\rho^{2}-1} \right\|_{\bullet}
    \le \exp\left( \left\| \frac{1}{\bm\rho^{2}-1} \right\|_{\ell^{1}} \right).
$$
Finiteness of the exponent requires $\rho^{-2} \in \ell^1$ or, equivalently, $(m^{-2\theta})_{m\in\mathbb{N}} \in\ell^1$.
This holds if $\theta > \frac12$.

\subsubsection*{Condition~\ref{itm:omega_lp}.}
Since $\bm\omega = (\bm\xi \bar{\bm\omega})^{p/(2-p)}$, the condition $\frac{\bm\omega^{(2-p)/p}}{\bar{\bm\omega}} \in \ell^p$ reduces to $\bm\xi \in \ell^1$.
Using the product structure of $\bm\xi_{\bm\nu} = \bm\rho^{-\kappa \bm\nu}$ and employing Lemma~\ref{lem:product_bound}, we can bound
$$
    \sum_{\bm\nu\in\mathbb{N}^M} \bm\xi_{\bm\nu}
    = \prod_{m=1}^{M} \sum_{n\in\mathbb{N}} \bm\rho_m^{-\kappa n} = \prod_{m=1}^{M} \frac{1}{1 - \bm\rho_m^{-\kappa}}
    = \prod_{m=1}^{M} \bigg(1 + \frac{1}{\bm\rho_m^{\kappa}-1}\bigg)
    \le \exp\bigg(\sum_{m=1}^M \frac{1}{\bm\rho_m^{\kappa}-1}\bigg)
    .
$$
Finiteness of the exponent $\|\frac{1}{\bm\rho^{\kappa}-1}\|_{\ell^1}$ requires $\bm\rho^{-\kappa}\in\ell^1$.
This holds if  $(m^{-\kappa\theta})_{m\in\mathbb{N}}\in\ell^1$, i.e.\ if $\kappa\theta > 1$.

\subsubsection*{Condition~\ref{itm:inv_omega_l1}.}
Recall that $p<1$ and define $r := \frac{p}{2-p} \le 2$ and $\tilde{\bm\rho} := \bm\rho^{(1-\kappa)r}$.
Then, using $\bar{\bm\omega}_{\bm\nu} = \ubar{\bm\omega}_{\bm\nu}^{-1}\bm\rho^{\bm\nu}$ and $\bm\xi_{\bm\nu} := \bm\rho^{-\kappa\bm\nu}$, we can rewrite $\bm\omega_{\bm\nu} = \bar{\bm\omega}^r_{\bm\nu} \bm\xi_{\bm\nu}^r = \ubar{\bm\omega}^{-r}_{\bm\nu} \tilde{\bm\rho}^{\bm\nu}$ and conclude that
$$
    \frac{\ubar{\bm\omega}_{\bm\nu}^2}{\bm\omega_{\bm\nu}}
    = \ubar{\bm\omega}_{\bm\nu}^{2+r} \tilde{\bm\rho}^{-\bm\nu}
    \le \ubar{\bm\omega}_{\bm\nu}^{4} \tilde{\bm\rho}^{-\bm\nu} .
$$
Using the product structure of this upper bound and the inequality $2n+1 \le 3n$ for $n \ge 1$, we conclude that
\begin{align}
    \sum_{\bm\nu\in\mathbb{N}^M} \frac{\ubar{\bm\omega}_{\bm\nu}^2}{\bm\omega_{\bm\nu}}
    &\le \sum_{\bm\nu\in\mathbb{N}^M} \ubar{\bm\omega}_{\bm\nu}^{4} \tilde{\bm\rho}^{-\bm\nu} \\
    &= \prod_{m=1}^M\sum_{n=0}^\infty (2n+1)^2 \tilde{\bm\rho}^{-n}_m \\
    &\le \prod_{m=1}^M\bigg(1 + 9\sum_{n=1}^\infty n^2 \tilde{\bm\rho}_m^{-n}\bigg) \\
    &= \prod_{m=1}^M\bigg(1 + 9 \frac{\tilde{\bm\rho}_m(\tilde{\bm\rho}_m+1)}{(\tilde{\bm\rho}_m-1)^3} \bigg) .
\end{align}
Finally, using Lemma~\ref{lem:product_bound}, we see that there exists a constant $C > 0$ such that
$$
    \prod_{m=1}^M\bigg(1 + 9 \frac{\tilde{\bm\rho}_m(\tilde{\bm\rho}_m+1)}{(\tilde{\bm\rho}_m-1)^3} \bigg)
    \le \exp\bigg(C\sum_{m=1}^M \tilde{\bm\rho}_m^{-1}\bigg) .
$$
Since the exponent is bounded by $C \zeta((1-\kappa)r\theta))$, finiteness requires $(1-\kappa)r\theta > 1$.

\subsubsection*{Combined conditions.}
Combining the above results, the assumptions \ref{asn:level_error}--\ref{asn:work_bound} are satisfied if the parameters satisfy the following constraints:
\begin{itemize}
    \item $\bm\rho_m = cm^\theta$ with $c\in(1, \frac{a_0}{\zeta(\eta - \theta)})$ and $\theta \in (\frac12, \eta -1)$,
    \item $\bm\omega = (\bm\xi\bar{\bm\omega})^{p/(2-p)}$ with $\bm\xi_{\bm\nu} = \bm\rho^{-\kappa\bm\nu}$ and $\kappa \in (\frac1\theta, 1 - \frac{1}{\theta}\frac{2-p}{p})$, and
    \item $p\in(0, 1)$.
\end{itemize}
Before discussing the numerical experiments, a few notes are in order.
\begin{enumerate}
    \item Note that $1 < c < \frac{a_0}{\zeta(\eta - \theta)}$ requires $a_0 > \zeta(\eta - \theta)$.
    \item $\frac1\theta < \kappa < 1 - \frac{1}{\theta}\frac{2-p}{p}$ implies $0 < \kappa < 1$ and requires $\theta > \frac2p > 2$.
    \label{itm:theta_bound_by_p}
    \item Since $\frac12 < \theta < \eta - 1$ implies $\eta > \theta + 1$, point~(\ref{itm:theta_bound_by_p}) entails $\eta > 3$.
\end{enumerate}
Assuming $\eta>3$ is indeed a harsh requirement.
Note, however, that this requirement follows under the choice $\bm\omega = (\bm\xi\bar{\bm\omega})^{p/(2-p)}$ with $\bm\xi_{\bm\nu} = \bm\rho^{-\kappa\bm\nu}$.
This is an extreme choice of $\bm\omega$ that mirrors the exponential growth of the maximally growing weight sequence $\bar{\bm\omega}$ (cf.~Theorem~\ref{thm:legendre_coefficient_bounds}) and, naturally, requires high regularity.
In contrast, we could consider the other extreme
\begin{equation}
\label{eq:omega_kappa}
    \bm\omega_{\bm\nu} := \prod_{m=1}^M (2\bm\nu_{m} + 1)^{1 + \bm\kappa_m}
    \quad\text{for}\quad
    \bm\kappa_{m} := 2\,\log_2(m+1)
\end{equation}
that mirrors the polynomial growth of the minimally growing weight sequence $\ubar{\bm\omega}$ (cf.~Theorem~\ref{thm:sparse_sample_complexity} and its proof).
Since $\bm\omega \lesssim (\bm\xi\bar{\bm\omega})^{p/(2-p)}$, condition~\ref{itm:omega_lp} follows trivially.
Moreover, $\big(\frac{\ubar{\bm\omega}^2}{\bm\omega}\big)_{\bm\nu} = \prod_{m=1}^M (2\bm{\nu}_m + 1)^{-\bm\kappa_m}$ and
\begin{align}
    \sum_{\nu\in\mathbb{N}^M} \frac{\ubar{\bm\omega}_{\bm\nu}^2}{\bm\omega_{\bm\nu}}
    = \sum_{\nu\in\mathbb{N}^M} \prod_{m=1}^M (2\bm{\nu}_m + 1)^{-\bm\kappa_m}
    = \prod_{m=1}^M \sum_{n=0}^\infty (2n + 1)^{-\bm\kappa_m}
    = \prod_{m=1}^M (1 - 2^{-\bm{\kappa}_m}) \zeta(\bm\kappa_m) \,.
\end{align}
Inserting the well-known bound $\zeta(\kappa) \le 1 + 2\cdot 2^{-\kappa}$ for $\kappa > 1$, Lemma~\ref{lem:product_bound} bounds
$$
    \prod_{m=1}^M (1 - 2^{\bm\kappa_m}) \zeta(\bm\kappa_m)
    \le \prod_{m=1}^M (1 + 2^{-\bm\kappa_m})
    \le \exp\bigg(\sum_{m=1}^M 2^{-\bm\kappa_m}\bigg)
    = \exp\bigg(\sum_{m=1}^M \frac{1}{m^2}\bigg)
    = \exp\bigg(\frac{\pi^2}{6}\bigg) \,.
$$

%% file: content/numerics.tex
\section{Numerical experiments}
\label{sec:experiments}

Similar to~\cite[Section 6]{Bouchot2017}, we consider the stationary diffusion problem~\eqref{eq:darcy} on the spatial domain $D = [0,1]^d$ with $d=1$.
We discretize this domain with a uniform mesh, using first-degree Lagrange finite elements to define the ansatz space for the numerical solution of $u$. 
We consider parametric diffusion coefficient~\eqref{eq:diffusionCoefficientReal} of the form \eqref{eq:example_diffcoeff} with decay exponent $\eta =2$, truncated after $M=6$ terms.
To ensure positivity of $a(x, \bm{y})$ and hence well-posedness of the PDE, we assume that the constant coefficient satisfies $a_{0} \ge \zeta(2) = \sum_{m=1}^{\infty}m^{-2}$.
The source term in~\eqref{eq:darcy} is given by $f \equiv 100$.
As a QoI $g\colon U \to\mathbb{R}$ we consider
\begin{equation*}
    g(\bm{y}) = \int_D u(x,\bm{y}) \, dx, \quad \bm{y} \in U.
\end{equation*}

To evaluate the effectiveness of our approach, we vary the number of levels $L = 1, \dots, 5$ and choose $h_0$ to keep the mesh width on the finest level constant at $h_L = 2^{-10}$.
For every choice of $L$, we approximate the QoI using a sequence of $10$ sample sizes to obtain a curve of ``error'' against ``required work''.
The computation of the ``error'' is detailed in subsection~\ref{sec:error}.
The ``required work'' is proportional to the cumulative time required to solve the FE problems over all levels $l\in[L]$ and for all samples.
The $10$ different sample sizes on level $l\in[L]$ are given by
$$
    n_{l,k} = c 2^{k/2} 2^{(L-l)/2} ,\quad\text{with}\quad k\in[10],
$$
yielding $10$ different estimators with increasing work.
The choice $n_{l,k} \propto 2^{(L-l)/2}$ is given by equation~\eqref{eq:level_wise_sample_size} with $\alpha\beta \le 2$, as derived in Section~\ref{sec:application}.
Allowing the choice $L=1$, we also compare the results of the multi-level method with those of the single-level method.
Each plot includes the ``FEM error'' as a reference for the minimally attainable error.

In our experiments we compare the model classes $\mathcal{M}_{\infty, r,\bm{\omega}}^1$ and $\mathcal{M}_{\infty, r,\bm{\omega}}^2$, defined in equations~\eqref{eq:MRromega_sparse} and~\eqref{eq:MRromega}, for two different weight sequences $\bm\omega$.
Assuming a polynomial growth for the radius sequence $\bm\rho_{m} = cm^{\beta}$, as in section~\ref{sec:verification}, leads us to choose $\bm\omega$, according to~\eqref{eq:omega_xi}, as
\begin{equation}
    \label{eq:weight_sequence_stronger_decay}
    \bm\omega_{\bm\nu} \propto \prod_{m=1}^M \sqrt{2\bm\nu_m +1} m^{3/4\bm\nu_m} .
\end{equation}
This sequence grows significantly faster than the one used in~\cite{Bouchot2017}.
For a finite expansion ($M<\infty$), we can choose a constant radius sequence $\bm\rho$, yielding the same weight sequence as in~\cite{Bouchot2017}, namely,
\begin{equation}
    \label{eq:weight_sequence_weaker_decay}
    \bm\omega_{\bm\nu} \propto \prod_{m=1}^M \sqrt{2\bm\nu_m +1} \,.
\end{equation}
The resulting $\bm\omega$ is equivalent to the minimally growing sequence $\ubar{\bm\omega}$ and similar to the weight sequence considered in~\eqref{eq:omega_kappa}.

Approximation in the model class $\mathcal{M}^1_{\infty,r,\bm\omega}$ was performed using the \emph{sparse alternating least squares} (SALS) algorithm~\cite[Section~4]{TrunschkeEigelNouy2025}, cf. Appendix~\ref{app:tensors}.
The choice of unbounded rank ($R = \infty$) is justified by the fact that the SALS algorithm adapts the rank based on the sparsity $r$.
Contrary to the definition of $\mathcal{M}_{R,r,\bm\omega}$, SALS does not ensure that the weighted sparsity of the core tensor is bounded by $r$ (see the definition of $\mathcal{C}_{Q,r,\bm\omega}$ in~\eqref{eq:model_class_1_core}).
However, since the algorithm chooses $r$ using cross-validation, and since the training set size $n_l$ is chosen according to the sparsity, we expect appropriate error bounds even without modifying the original algorithm.
This expectation is, indeed, validated by our experiments.

Approximation in the model class $\mathcal{M}^2_{\infty,r,\bm\omega}$ was performed using the \emph{semi-sparse alternating least squares} (SSALS) algorithm~\cite[Section~5]{TrunschkeEigelNouy2025}, cf. Appendix~\ref{app:tensors}.
Just like SALS, this algorithm is rank-adaptive by design\footnote{The algorithm uses the rank-adaptation strategy proposed in~\cite{Grasedyck2019SALSA}.} and chooses the sparsity based on cross-validation.
Again, we argue that rank-adaptivity justifies the choice $R=\infty$ and, by the same arguments as above, we expect appropriate error bounds even without modifying the original algorithm.

\subsection{Error computation}
\label{sec:error}

To compute the error, we estimate the relative $L^2$ residual by the following expression
\begin{equation*}
    \mathrm{RMSE}
    := \sqrt{
    \frac{\sum_{m \in [M_{\mathrm{test}}]} \vert h(\bm{y}^{(m)}) - G^{(m)}\vert^2}
    {\sum_{m \in [M_{\mathrm{test}}]} \vert G^{(m)}\vert^2}},
\end{equation*}
where $(\bm{y}^{(m)})_m$ are independently drawn points drawn according to the uniform distribution and
\begin{equation*}
    G^{(m)}= \int_D u^{\mathrm{test}}(x,\bm{y}^{(m)})\, dx.
\end{equation*}
In this case, $u^{\mathrm{test}}(\cdot,\boldsymbol{y}^{(m)})$ denotes the FE solution computed with quadratic Lagrange elements ($\operatorname{deg}=2$) on a uniform grid of width $h = 2^{-14}$. Moreover, the test set size is chosen as $M_{\mathrm{test}} = 1000$. At every point, we report the median RMSE over $20$ independent trials.

Along \eqref{eq:work_scale_Galerkin_dim} and \eqref{eq:Galerkin_dim_scale_spacial_dim}, we compare the RMSE against the estimated work:
$$
    \tau_L^{\mathrm{ML}} = h_0^{-1} \Bigl( N_1 2 + \sum_{l=2}^L N_l (2^{l} + 2^{l-1}) \Bigr) .
$$

\subsection{Results}

\begin{figure}
    \centering
    \includegraphics[width=\linewidth]{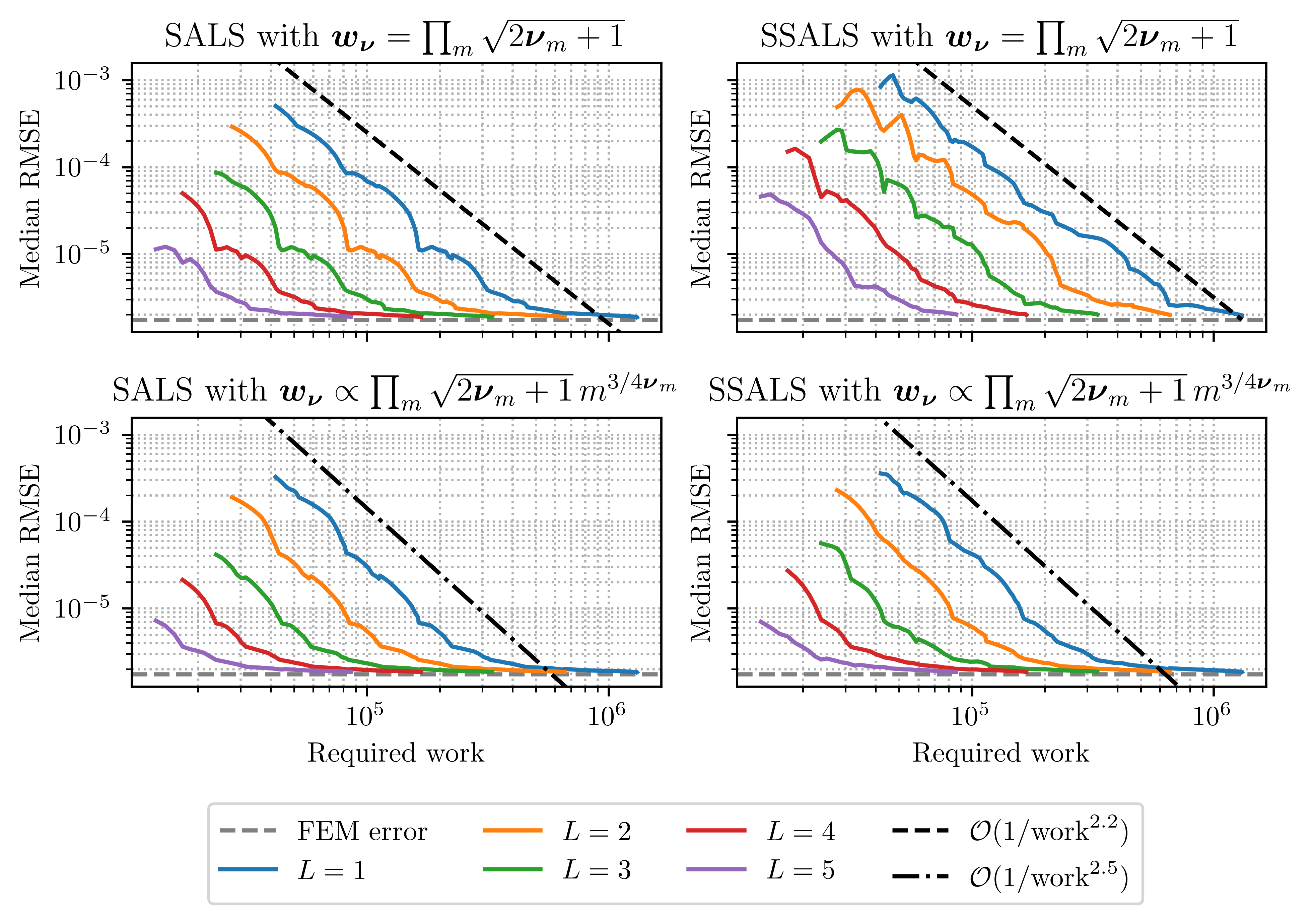}
    \caption{Median RMSE vs. required work for SALS and SSALS with weaker (top) and stronger (bottom) decaying weights. Both methods converge similarly, stronger decay improves accuracy, and increasing the number of levels $L$ enhances the multilevel efficiency. Dashed lines indicate reference rates.}
    \label{fig:salsr1}
\end{figure}

Figure~\ref{fig:salsr1} shows the median RMSE of the multilevel methods for the described experiment in relation to the estimated required work. Each of the two algorithms, SALS and SSALS, was performed with both weight sequences \eqref{eq:weight_sequence_stronger_decay} and \eqref{eq:weight_sequence_weaker_decay}, resulting in a total of 4 plots.

Across all configurations, increasing the number of levels leads to a systematic shift in the error-work curves: with a fixed target accuracy, multilevel estimators require significantly less work than their single-level counterparts. This behavior is fully consistent with the theoretical complexity results from Theorem~\ref{thm:sl_error_and_work} and Theorem~\ref{thm:ml_error_and_work}, where the application of the multilevel approach predicts a reduction in computational costs. In particular, for increasing $L$, the lower bound of FEM accuracy is approximated with significantly estimated computational effort.

The numerical experiments show that SALS and SSALS exhibit essentially the same convergence behavior in terms of RMSE decay. This is in line with the theory, which predicts comparable error bounds and thus similar approximation quality for both variants. We would like to point out that the advantage of SSALS, as described in ~\cite{TrunschkeEigelNouy2025}, lies mainly in its efficient calculation. This is not reflected in the RMSE diagrams, where only the degrees of freedom required by the FEM calculation are shown on the x-axis.

Furthermore, the RMSE decays faster for the stronger decaying weight sequence~\eqref{eq:weight_sequence_stronger_decay} than for~\eqref{eq:weight_sequence_weaker_decay}.
This matches our expectation, since the weight sequence better reflects the decay of the coefficients.

%% file: content/appendix/tensors.tex
\section{Sparse tensor formats and the SALS/SSALS algorithms}
\label{app:tensors}

This appendix recalls the notion of sparse tensor networks used as model class in the manuscript.
We use notation consistent with Section~\ref{sec:model classes} and follow the presentation in~\cite{TrunschkeEigelNouy2025}.

\subsection{Tensor preliminaries}
\label{subsec:tensor_preliminaries}

Let $M\in\mathbb N$ be the order, $d_1,\dots,d_M\in\mathbb N$ be mode sizes, and
\[
    \Lambda := [d_1]\times\cdots\times[d_M],\qquad [d] := \{1,\dots,d\}.
\]

For $\mathcal A\in\mathbb R^{d_1\times\cdots\times d_M}$ we define vectorisation using the
lexicographic index map
\[
    \phi(i_1,\dots,i_M) := 1+\sum_{m=1}^M (i_m-1)\prod_{j=1}^{m-1} d_j,
\]
with empty product set to $1$ and let
\[
    \big(\operatorname{vec}(\mathcal A)\big)_{\phi(i_1,\dots,i_M)} := \mathcal A_{i_1,\dots,i_M}.
\]
For $k\in[M]$, define the row/column index maps
\[
    \phi_{\le k}(i_1,\dots,i_k) := 1+\sum_{m=1}^k (i_m-1)\prod_{j=1}^{m-1} d_j, \qquad \phi_{>k}(i_{k+1},\dots,i_M) := 1+\sum_{m=k+1}^M (i_m-1)\prod_{j=k+1}^{m-1} d_j,
\]
and the $k$-unfolding $\operatorname{unfold}_k(\mathcal A)\in\mathbb R^{(\prod_{j=1}^k d_j)\times(\prod_{j=k+1}^M d_j)}$ by
\[
    \big(\operatorname{unfold}_k(\mathcal A)\big)_{\phi_{\le k}(i_1,\dots,i_k),\,\phi_{>k}(i_{k+1},\dots,i_M)} := \mathcal A_{i_1,\dots,i_M}.
\]
With these conventions, $\operatorname{vec}(\mathcal A)=\operatorname{vec}(\operatorname{unfold}_k(\mathcal A))$ holds.
We identify coefficient vectors in $\mathbb R^{d_1\cdots d_M}$ with tensors $\mathcal A\in\mathbb R^{d_1\times\cdots\times d_M}$ by vectorisation.
Given component tensors $\mathcal A^{(k)}\in\mathbb R^{r_{k-1}\times d_k\times r_k}$, $r_0=r_M=1$, the tensor train (TT) representation is
\[
    \mathcal A = \mathcal A^{(1)}\circ\cdots\circ \mathcal A^{(M)},
\]
where $\circ$ denotes the contraction along the last dimension of the left tensor and the first dimension of the right tensor.
The tuple $(r_1, \ldots, r_{M-1})$ is called TT rank.
A component tensor $\mathcal A^{(k)}$ is called left-orthogonal if $\mathrm{unfold}_{2}(\mathcal A^{(k)})^\top \mathrm{unfold}_{2}(\mathcal A^{(k)})=I$,
and right-orthogonal if $\mathrm{unfold}_{1}(\mathcal A^{(k)})\mathrm{unfold}_{1}(\mathcal A^{(k)})^\top=I$.
A TT representation $\mathcal A=\mathcal A^{(1)}\circ\cdots\circ \mathcal A^{(M)}$ has \emph{core position $k$} if $\mathcal A^{(1)},\dots,\mathcal A^{(k-1)}$ are left-orthogonal and $\mathcal A^{(k+1)},\dots,\mathcal A^{(M)}$ are right-orthogonal with no orthogonality requirement imposed on $\mathcal A^{(k)}$.

\subsection{Weighted sparsity}
To define weighted sparse tensors, let $\omega=(\omega_\nu)_{\nu\in\Lambda}$ with $\omega_\nu>0$ and for $\mathbf c=(c_\nu)_{\nu\in\Lambda}$ define
\[
    \|\mathbf c\|_{\ell^0_\omega} := \sum_{\nu\in\operatorname{supp}(\mathbf c)}\omega_\nu^2, \qquad \|\mathbf c\|_{\ell^p_\omega} := \|\omega\odot \mathbf c\|_{\ell^p}
\]
with componentwise multiplication $\odot$.
We call $\mathbf c$ $(\omega,s)$-sparse if $\|\mathbf c\|_{\ell^0_\omega}\le s$.
For a sparse matrix $X$, let $\operatorname{row}(X)$, $\operatorname{col}(X)$, $\operatorname{data}(X)$ be tuples or rows, columns and values of nonzeros, respectively, and $\operatorname{coo}(r,c,d)$ the sparse matrix with row tuple $r$, column tuple $c$, and values $d$.
Also define $\operatorname{range}(r)=(1,\dots,r)$, $\operatorname{ones}(r)=(1,\dots,1)$, and $\dim(x)$ as vector length.

For a sparse matrix $X\in\mathbb R^{n\times m}$, let $s=\operatorname{unique}(\operatorname{row}(X))$ and $r=\dim(s)$.
Define
\[
    Q:=\operatorname{coo}(s,[r],\mathbf 1)\in\mathbb R^{n\times r}, \qquad C:=Q^\top X\in\mathbb R^{r\times m}.
\]
Then $X=QC$, $Q^\top Q=I_r$, and the columns of $Q$ are standard basis vectors.

\subsection{Sparsely canonicalized TT}
\label{subsec:sctt}

We call a tensor $\mathcal T$ \emph{$s$-sparse} if $\|\operatorname{vec}(\mathcal T)\|_{\ell^0}\le s$.
A TT representation with core position $k$
\begin{equation}\label{eq:TTk}
    \mathcal A = \mathcal U^{(1)}\circ\cdots\circ \mathcal U^{(k-1)} \circ \mathcal C \circ \mathcal V^{(k+1)}\circ\cdots\circ \mathcal V^{(M)}
\end{equation}
is called sparsely canonicalized (SCTT) if
\begin{enumerate}
    \item $\mathcal U^{(j)}\in\{0,1\}^{r_{j-1}\times d_j\times r_j}$ is left-orthogonal and $r_j$-sparse for $1\le j<k$,
    \item $\mathcal V^{(j)}\in\{0,1\}^{r_{j-1}\times d_j\times r_j}$ is right-orthogonal and $r_{j-1}$-sparse for $k<j\le M$,
    \item $\mathcal C\in\mathbb R^{r_{k-1}\times d_k\times r_k}$ is $\min\{r_{k-1},r_k\}$-sparse.
\end{enumerate}

Every $s$-sparse tensor admits such a representation for any core position $k$ with TT ranks bounded by $s$.
Moreover, with
\[
    Q = \mathrm{unfold}_{k}\!\big(\mathcal U^{(1)}\circ\cdots\circ\mathcal U^{(k-1)}\big) \otimes I_{d_k}\otimes \mathrm{unfold}_{1}\!\big(\mathcal V^{(k+1)}\circ\cdots\circ\mathcal V^{(M)}\big)^\top,
\]
we have $\operatorname{vec}(\mathcal A)=Q\operatorname{vec}(\mathcal C)$ and the isometry
\[
    \|\mathcal A\|_{\ell^q_\omega} = \|\mathcal C\|_{\ell^q_{\omega_Q}}, \qquad \omega_Q:=Q^\top\omega, \quad q\in[0,\infty],
\]
using $\|\mathcal T\|_{\ell^q_\omega} := \|\operatorname{vec}(\mathcal T)\|_{\ell^q_\omega}$.
Algorithm~\ref{alg:sparse_canon} depicts the steps to obtain a SCTT.

\begin{algorithm}[h!]
\caption{Sparse canonicalisation (adapted from \cite[Alg.~1]{TrunschkeEigelNouy2025})}
\label{alg:sparse_canon}
\begin{algorithmic}[1]
    \Statex \textbf{Input:} TT $\mathcal A=\mathcal A^{(1)}\circ\cdots\circ\mathcal A^{(M)}$, core position $k$.
    \Statex \textbf{Output:} SCTT with core at $k$.
    
    \State $C^{(0)} \gets I$
    \For{$j=1,\dots,k-1$}
      \State $X^{(j)} \gets \mathrm{unfold}_2(C^{(j-1)}\circ\mathcal A^{(j)})$
      \State sparse QC: $X^{(j)}=Q^{(j)}C^{(j)}$
      \State $\mathrm{unfold}_2(\mathcal U^{(j)}) \gets Q^{(j)}$
    \EndFor
    
    \State $C^{(M+1)} \gets I$
    \For{$j=M,\dots,k+1$}
      \State $X^{(j)} \gets \mathrm{unfold}_1(\mathcal A^{(j)}\circ C^{(j+1)})$
      \State sparse QC: $(X^{(j)})^\top=(Q^{(j)})^\top (C^{(j)})^\top$
      \State $\mathrm{unfold}_1(\mathcal V^{(j)}) \gets Q^{(j)}$
    \EndFor
    
    \State $\mathcal C \gets C^{(k-1)} \circ \mathcal A^{(k)} \circ C^{(k+1)}$
    \State \Return $\mathcal U^{(1)}\circ\cdots\circ\mathcal U^{(k-1)}\circ\mathcal C\circ\mathcal V^{(k+1)}\circ\cdots\circ\mathcal V^{(M)}$
    \end{algorithmic}
\end{algorithm}

\subsection{ALS for sparse tensor model classes}
\label{subsec:sals}

Let $\Gamma$ be the parameter domain equipped with a target measure $\mu$ and let $w:\Gamma\to(0,\infty)$ be a weight function such that the sampling measure is $\rho:=w^{-1}\mu$, cf.~\eqref{eq:min_emp}.
Let $y_i\in\Gamma$, $i=1,\dots,n$, be i.i.d.\ samples drawn from $\rho$ and let $g:\Gamma\to\mathbb R$ denote the target quantity to be approximated.
Let $(B_\nu)_{\nu\in\Lambda}$ be the product basis and define the basis-evaluation tensor $B(y)\in\mathbb R^{d_1\times\cdots\times d_M}$ by $(B(y))_\nu := B_\nu(y)$.
For a coefficient tensor $\mathcal A\in\mathbb R^{d_1\times\cdots\times d_M}$ we use the evaluation
\[
    \langle \mathcal A, B(y)\rangle_{\mathrm{Fro}}=\sum_{\nu\in\Lambda}\mathcal A_\nu B_\nu(y).
\]
Define $F\in\mathbb R^n$ and the design operator $M$ by
\[
    F_i := \sqrt{w(y_i)}g(y_i), \qquad (M\mathcal A)_i := \sqrt{w(y_i)}\langle \mathcal A,B(y_i)\rangle_{\mathrm{Fro}}.
\]
Then empirical least squares over the model class $\mathcal{M}_{R,r,\omega}^1$ reads
\[
    \min_{\mathcal A\in \mathcal M^1_{R,r,\omega}}\|F-M\mathcal A\|_2^2.
\]
Writing the coefficient tensor $\mathcal A$ in SCTT form with core position $k$ as in~\eqref{eq:TTk}, we identify the core with its vectorisation $c := \operatorname{vec}(\mathcal C)\in\mathbb R^{r_{k-1}d_k r_k}$ and use matrix $Q\in\mathbb R^{(d_1\cdots d_M)\times(r_{k-1}d_k r_k)}$ such that $\operatorname{vec}(\mathcal A)=Qc$.
The ALS microstep therefore reads
\[
    \min_{c\in\mathcal C_{Q,r,\omega}} \|F - M(Qc)\|_2^2, \qquad \mathcal C_{Q,r,\omega} := \{c:\|c\|_{\ell^0_{\omega_Q}}\le r\}, \qquad \omega_Q := Q^\top\omega.
\]
A convex relaxation is given by the weighted LASSO
\[
    \min_c \|F-MQc\|_2^2+\lambda\|\omega_Q\odot c\|_1, \quad \lambda>0
\]
With $d=\omega_Q\odot c$, this is equivalent to
\[
    \min_d \|F-MQ(\omega_Q^{-1}\odot d)\|_2^2+\lambda\|d\|_1.
\]
This is a standard $\ell^1$-regularised optimization problem and we select $\lambda$ by cross-validation.
Iterating sequentially over all core positions $k\in[M]$ yields the SALS algorithm as a sparse, rank-adaptive ALS variant depicted in Algorithm~\ref{alg:sals}.

\begin{algorithm}[h!]
\caption{Sparse Alternating Least Squares (SALS)}
\label{alg:sals}
\begin{algorithmic}[1]
\Statex \textbf{Input:} samples $\{(y_i,g(y_i))\}_{i=1}^n$, basis data, weights $\omega$, initial TT $\mathcal A$.
\Statex \textbf{Output:} coefficient tensor $\mathcal A\in\mathcal M^1_{R,r,\omega}$.

\While{not converged}
    \For{$k=1,\dots,M$}
        \State compute sparse canonicalisation with core at $k$ (Alg.~\ref{alg:sparse_canon})
        \State form $Q$ and $\omega_Q=Q^\top\omega$
        \State update core by LASSO (choose $\lambda$ via cross-validation)
    \EndFor
\EndWhile
\State \Return $\mathcal A$
\end{algorithmic}
\end{algorithm}

The semi-sparse model class\footnote{denoted $\widetilde{\mathcal M}_{R,r,\omega}$ in~\cite{TrunschkeEigelNouy2025}} $\mathcal M^2_{R,r,\omega}$ is obtained by replacing the sparse QC with an $\omega$-orthogonal QC
\[
    Q^\top \operatorname{diag}(\omega)\,Q \text{ diagonal}.
\]
Then define
\[
    \omega_Q^2 := \operatorname{diag}(Q^\top \operatorname{diag}(\omega^2)\,Q), \qquad \mathcal C_{Q,r,\omega}^{(2)}:=\{C:\|C\|_{\ell^0_{\omega_Q}}\le r\}.
\]
The resulting SSALS algorithm differs from SALS only by this decomposition step.
In contrast to SALS, intrinsic rank-adaptivity is generally lost.
It can be regained by standard algorithms, such as the one used in~\cite{Grasedyck2019SALSA}.

%% file: content/appendix/proof_summability_2.tex
\section{Proof of Theorem \ref{thm:legendre_coefficient_bounds}}
\label{app:proof:thm:legendre_coefficient_bounds}

Let $v : B_{\mathbb{C}}(0, \bm\rho) \to X$ be holomorphic such that 
$J := \|v\|_{L^\infty(B_{\mathbb{C}}(0, \bm\rho); X)} < \infty$.
Denote by $\bm v$ the power series coefficients of $v$ and by $\hat{\bm v}$ the Legendre series coefficients of $v$.
To bound $\hat{\bm v}$, we first bound $\bm v$ and then use the relation between both coefficient sequences to arrive at the desired bound.
The summability of the power series coefficients $\bm v$ follows from the subsequent generalisation of Cauchy's theorem.

\begin{lemma}[Lemma~2.4 in~\cite{Cohen2010}]
\label{lem:cauchy_banach}
    Let $\bm\rho\in(1,\infty)^M$, $X$ be a Banach space and $v : B_{\mathbb{C}}(0, \bm\rho) \to X$ be holomorphic such that 
    $J := \|v\|_{L^\infty(B_{\mathbb{C}}(0, \bm\rho); X)} < \infty$.
    Then the (multivariate) power series coefficients $\bm{v}\in X^{\mathbb{N}}$ of $v$ satisfy
    \begin{equation}
        \|\boldsymbol{v}_{\bm n}\|_X \le J \bm\rho^{-\bm n} .
    \end{equation}
\end{lemma}

The relation between $\hat{\bm v}$ and $\bm v$ is presented in the subsequent two lemmas.
\begin{lemma}
\label{lem:legendre_coefficients}
    Let $\bm\rho\in(1,\infty)^M$, $X$ be a Banach space and $v : B_{\mathbb{C}}(0, \bm\rho) \to X$ and $\bm{v}$ be the (multivariate) power series coefficients of $v$.
    Then the (multivariate) Legendre polynomial coefficients $\hat{\bm v}$ of $v$ are given by
    $$
        \hat{\bm{v}}_{\bm m} = \sum_{\bm{n}\in\mathbb{N}^M} \bm{\alpha}_{\bm{m}, \bm{n}} \bm{v}_{\bm{n}}
    $$
    where $\bm{\alpha}_{\bm{m},\bm{n}}$ is defined by
    \begin{equation}
        \bm{\alpha} := \alpha^{\otimes M}
        \quad\text{with}\quad
        \alpha_{m,k} := \begin{cases}
            \frac{(2m + 1)^{1/2}(m+2n)!}{2^{m+2n}n!(\tfrac{3}{2})_{m+n}} \,, & k = m + 2n\,, \\
            0\,, & \text{otherwise} \,,
        \end{cases}
    \end{equation}
    with the Pochhammer symbol $(a)_0 = 1$ and $(a)_{k+1} = (a)_k(a+k)$.
\end{lemma}
\begin{proof}
    We prove this claim by induction over $M$.
    For the univariate case $M=1$, a proof that the $m$\textsuperscript{th} Legendre series coefficient is $\sum_{n\in\mathbb{N}} \alpha_{m,n} \bm{v}_n$ is given in~\cite{Iserles2010}.
    Now, let $M'=M+1$ and let $L_{\bm{m}}$ denote the $\bm{m}$\textsuperscript{th} normalized product Legendre polynomial.
    Using the induction assumption, we can write
    \begin{align}
        v(y, \bm{y})
        &= \sum_{\bm{m}\in\mathbb{N}^{M}}
        \bigg(\sum_{m\in\mathbb{N}\vphantom{\mathbb{N}^M}}
        \bm{v}_{(m,\bm{m})} y^m\bigg) \bm{y}^{\bm{m}} \\
        &= \sum_{\bm{m}\in\mathbb{N}^{M}}
        \tilde{\bm{v}}_{\bm{m}}(y) \bm{y}^{\bm{m}} \\
        &= \sum_{\bm{m}\in\mathbb{N}^{M}}
        \bigg(\sum_{\bm{n}\in\mathbb{N}^M} \bm{\alpha}_{\bm{m},\bm{n}} \tilde{\bm{v}}_{\bm{n}}(y)\bigg) L_{\bm{m}}(\bm{y})
    \end{align}
    with
    $$
        \tilde{\bm{v}}_{\bm{n}}(y)
        = \sum_{m\in\mathbb{N}} \bm{v}_{(m,\bm{n})} y^m
        = \sum_{m\in\mathbb{N}} \bigg(\sum_{n\in\mathbb{N}} \alpha_{m,n} \bm{v}_{(n,\bm{n})}\bigg) L_m(y)
        \,.
    $$
    Combining both equations yields
    \begin{align}
        v(y, \bm{y})
        &= \sum_{\bm{m}\in\mathbb{N}^{M}}
        \Bigg(\sum_{\bm{n}\in\mathbb{N}^M} \bm{\alpha}_{\bm{m},\bm{n}} \Bigg(\sum_{m\in\mathbb{N}} \bigg(\sum_{n\in\mathbb{N}} \alpha_{m,n} \bm{v}_{(n,\bm{n})}\bigg) L_m(y)\Bigg)\Bigg) L_{\bm{m}}(\bm{y}) \\
        &= \sum_{\bm{m}\in\mathbb{N}^{M}} \sum_{m\in\mathbb{N}\vphantom{\mathbb{N}^M}}
        \bigg(
            \sum_{\bm{n}\in\mathbb{N}^M}
            \sum_{n\in\mathbb{N}\vphantom{\mathbb{N}^M}}
            \alpha_{m,n} \bm{\alpha}_{\bm{m},\bm{n}} \bm{v}_{(n,\bm{n})}
        \bigg) L_m(y) L_{\bm{m}}(\bm{y}) \\
        &= \sum_{\bm{m}'\in\mathbb{N}^{M'}}
        \Bigg(
            \sum_{\bm{n}'\in\mathbb{N}^{M'}}
            \bm{\alpha}_{\bm{m}',\bm{n}'} \bm{v}_{\bm{n}'}
        \Bigg) L_{\bm{m}'}(y, \bm{y})
        \,. \qedhere
    \end{align}
\end{proof}

\begin{lemma}
\label{lem:alpha_le_beta}
    Define $\bm{\beta} := \beta^{\otimes M}$ with
    $$
        \beta_{m,k} := \begin{cases}
            (2m+1)^{1/2} \,, & k = m + 2n\,, \\
            0\,, & \text{otherwise} \,.
        \end{cases}
    $$
    Then, $\bm\alpha \le \bm\beta$.
\end{lemma}
\begin{proof}
    Observe, that the Pochhammer symbol satisfies $(a)_k > (b)_k$ when $a > b$ and that $(1)_k = k!$.
    Consequently, $(\tfrac{3}{2})_k > (1)_k = k!$, implying
    \begin{equation}
    \label{eq:est_alpha_step_1}
        \alpha_{m, m + 2n}
        = \frac{(2m + 1)^{1/2}(m+2n)!}{2^{m+2n}n!(\tfrac{3}{2})_{m+n}}
        < \frac{(2m + 1)^{1/2}(m+2n)!}{2^{m+2n}n!(m+n)!}
        = (2m + 1)^{1/2} 2^{-(m+2n)} \binom{m+2n}{n}
        \,.
    \end{equation}
    For $n=0$, this immediately implies
    \[
        \alpha_{m, m}
        < (2m + 1)^{1/2}2^{-m} \le \beta_{m,m}.
    \]
    To bound \eqref{eq:est_alpha_step_1} further for $n\ge 1$, we observe that
    $$
        \binom{m+2n}{n}
        =
        \frac{(m+2n)!}{n!(m+n)!}
        = \bigg( \prod_{k=1}^m \frac{2n + k}{n + k} \bigg)\frac{(2n)!}{n!n!}
        \le 2^m \binom{2n}{n}
    $$
    and use Stirling's formula~\cite{Robbins1955}, $n! \in \big(\sqrt{2\pi}, e\big] \cdot n^{n+1/2} e^{-n}$, to calculate
    $$
        \binom{2n}{n} \in \big[\tfrac{\sqrt{2\pi}}{e^2}, \tfrac{e}{\pi} \big] \cdot n^{-1/2} 2^{2n} \,.
    $$
    Inserting the upper bound into the estimate~\eqref{eq:est_alpha_step_1} yields
    \begin{align}
        \alpha_{m, m + 2n}
        \le (2m + 1)^{1/2}  2^{-2n}\binom{2n}{n}        \le \frac{e}{\pi} (2m + 1)^{1/2}n^{-1/2} \le \beta_{m, m+2n} \,.
    \end{align}
\end{proof}

Having established all necessary building blocks, we can now proceed to prove Theorem~\ref{thm:legendre_coefficient_bounds}.
By Lemma~\ref{lem:legendre_coefficients} and the triangle inequality, we can bound
\begin{equation}
\label{eq:legendre_coefficient_bounds}
    \|\hat{\bm v}_{\bm m}\|_{X}
    = \Big\|\sum_{\bm k\in\mathbb{N}^M} \bm\alpha_{\bm m, \bm k}\bm v_{\bm k}\Big\|_{X}
    \le \sum_{\bm k\in\mathbb{N}^M} \bm\alpha_{\bm m, \bm k}\|\bm v_{\bm k}\|_{X}
    \,.
\end{equation}
Using Lemmas~\ref{lem:cauchy_banach} and~\ref{lem:alpha_le_beta}, we further bound
\begin{align}
    \|\hat{\bm v}_{\bm m}\|_{X}
    &\le \sum_{\bm k\in\mathbb{N}^M} \bm\alpha_{\bm m, \bm k}\|\bm v_{\bm k}\|_{X} \\
    &= \sum_{\bm n\in\mathbb{N}^M} \bm\alpha_{\bm m, \bm m + 2\bm n}\|\bm v_{\bm m + 2\bm n}\|_{X} \\
    &\le \sum_{\bm n\in\mathbb{N}^M} \|(2\bm m+1)^{1/2}\|_{\bullet} J{\bm\rho}^{-(\bm m + 2\bm n)} \\
    &= J \|2\bm m+1\|_{\bullet}^{1/2} \bm\rho^{-\bm m}\sum_{\bm n\in\mathbb{N}^M}  {\bm\rho}^{-2\bm n} \\
    &= J \bar{\bm\omega}_{\bm m}^{-1} \sum_{\bm n\in\mathbb{N}^M} \bm{\rho}^{-2\bm n}
    \,.
\end{align}
Finally, note that
$$
    \sum_{\bm n\in\mathbb{N}^M} \bm{\rho}^{-2\bm n}
    = \prod_{k=1}^M\Big(\sum_{n\in\mathbb{N}} \bm{\rho}_k^{-2n}\Big)
    = \prod_{k=1}^M\Big(\frac{\bm \rho_k^2}{\bm \rho_k^2 -1}\Big)
    = \Big\|\frac{\bm \rho^2}{\bm \rho^2 -1}\Big\|_{\bullet} .
$$
This proves the first claim.
The second and third claim follow directly from $\frac{1}{\bar{\bm\omega}} \le \frac{\ubar{\bm\omega}}{\bar{\bm\omega}} \le \frac{\bm\omega^{(2-p)/p}}{\bar{\bm\omega}}$ and the assumption $\frac{\bm\omega^{(2-p)/p}}{\bar{\bm\omega}} \in \ell^p$.

\section{Proof of Theorem \ref{thm:supremum_norm_bound}}
\label{app:proof:thm:supremum_norm_bound}

Let $u$ be the solution of the diffusion equation~\eqref{eq:darcy} with affine coefficients~\eqref{eq:diffusionCoefficientReal} satisfying~\eqref{eq:rho-UEA}.

Following~\cite{CohenDeVore2015}, we begin by proving that $u$ and $u_l$ are uniformly bounded on $B_{\mathbb{C}}(0, \bm\rho)$.

Using~\eqref{eq:rho-UEA}, we derive that 
$$
    a(x, \bm y)
    \ge a_0(x) - \sum_{m=1}^M \rho_m |a_m(x)|
    \ge \inf_{x\in D} a_0(x) - \sum_{m=1}^M \rho_m |a_m(x)|
    = c
$$
for all $(x,\bm y)\in D\times B_{\mathbb{C}}(0, \bm\rho)$.
This implies that
$$
    \mathcal{A}(u, v; \bm{y})
    = \int a(\bm{y}) \nabla u \nabla v \;\mathrm{d}\lambda
    \ge c \int \nabla u \nabla v \;\mathrm{d}\lambda
$$
is uniformly elliptic with constant $c$.
The Lax--Milgram theorem thus implies
$$
    \|u(\bm y)\|_V
    \le \tfrac1c \|f\|_{V^*}
$$
for all $\bm y\in B_{\mathbb{C}}(0, \bm\rho)$ and, consequently,
$
    \|u\|_{L^\infty(B_{\mathbb{C}}(0, \bm\rho); V)} \le \tfrac{1}{c} \|f\|_{V^*}
$ .
Since the restriction of $\mathcal{A}$ to $V_l\times V_l$ remains $c$-elliptic, and since since $f\in V^*\subseteq V_l^*$, Lax--Milgram also implies
$$
    \|u_l(\bm y)\|_{V}
    = \|u_l(\bm y)\|_{V_l}
    \le \tfrac1c \|f\|_{V_l^*}
    \le \tfrac1c \|f\|_{V^*}
    \,.
$$

To prove that $u$ is holomorphic from $B_{\mathbb{C}}(0, \bm\rho)$ to $V$, we introduce for any coefficient field $a$ the operator $B(a) : v\mapsto -\operatorname{div}_{\!x}(a\nabla_{\!x}v)$ mapping from $V$ to $V^*$ and decompose the map $\bm y\mapsto u(\bm y)$ into the chain of holomorphic maps
\begin{equation}
    \bm y
    \mapsto a(\bm y)
    \mapsto B(a(\bm y))
    \mapsto B(a(\bm y))^{-1}
    \mapsto B(a(\bm y))^{-1} f
    = u(\bm y)
    \,.
\end{equation}
The first map is holomorphic by definition and the second and last map are continuous linear maps and thereby also holomorphic.
The third map is the operator inversion which is holomorphic at any invertible $B$.
Since $B(a(\bm y))$ is invertible for every $\bm y\in B_{\mathbb{C}}(0, \bm\rho)$, the composition $\bm y \mapsto u(\bm y)$ is holomorphic.
The same argument applies, mutatis mutandis, to $u_l$.

%% file: content/appendix/proof_rip_transfer.tex
\section{Proof of Theorem~\ref{thm:semi-sparse_sample_complexity}}
\label{app:proof:thm:semi-sparse_sample_complexity}

As before, we note that that $\mathcal{M}^2_{\gamma} \simeq \mathcal{M}^2_{\infty, r(\gamma) ,\bm\omega}$ and, by Proposition~5.13 in~\cite{TrunschkeEigelNouy2025}, it holds that
$$
    \mathcal{A}
    = \{P_{\mathcal{M}^2_{\gamma}}^n g_l\} - \mathcal{M}^2_{\gamma}
    \subseteq \mathcal{M}^2_{\gamma} - \mathcal{M}^2_{\gamma}
    = \mathcal{M}^2_{\infty,2r(\gamma),\bm\omega}
    =: \mathcal{B} \,.
$$
Since $\operatorname{RIP}_{\mathcal{B}}(\delta)$ implies $\operatorname{RIP}_{\mathcal{A}}(\delta)$, it suffices to prove $\operatorname{RIP}_{\mathcal{B}}(\delta)$.
To do this, we use a slight generalisation of Theorem~5.12 in~\cite{TrunschkeEigelNouy2025}.
For this we first recall some definitions from that reference.
First, define
$$
    \operatorname{Cone}(\mathcal{A}) := \big\{ta \ \big|\ a\in \mathcal{A},\, t>0\big\}
    \quad\text{and}\quad
    U(\mathcal{A}) := \Big\{\tfrac{a}{\|a\|_{L^2}} \ \big|\ a\in \mathcal{A}\Big\}
$$
as well as the scale-invariant non-symmetric distance function
\begin{equation}
\label{eq:dsiL}
    d_{\mathrm{si}L^\infty_w}(\mathcal{A}, \mathcal{B})
    := \sup_{a\in \mathcal{A}} d_{L^\infty_w}\big(\operatorname{Cone}(a), U(\mathcal{B})\big)
    \quad\text{with}\quad
    d_{L^\infty_w}(\mathcal{A}, \mathcal{B})
    := \inf_{a\in \mathcal{A}}\inf_{b\in \mathcal{B}} \|a - b\|_{L^\infty_w}~.
\end{equation}

Before stating the theorem, we first prove the subsequent lemma that is used later to simplify the proof.
\begin{lemma}
\label{lem:simplified_dsi}
    Let $r\in(0,1)$. Then
    \begin{align*}
        \forall a\in\operatorname{Cone}(\mathcal{A})\ \exists\, b\in\operatorname{Cone}(\mathcal{B}) : \|a - b\|_{L^\infty_w} \le r \|a\|_{L^2}
        \quad\Rightarrow\quad
        d_{\mathrm{si}L^\infty_w}(\mathcal{A}, \mathcal{B}) \le \tfrac{r}{1-r}
    \end{align*}
\end{lemma}
\begin{proof}
    Observe that proving $d_{\mathrm{si}L^\infty_w}(\mathcal{A}, \mathcal{B}) \le \tfrac{r}{1-r}$ is equivalent to proving
    \begin{equation}
        \forall a\in\operatorname{Cone}(\mathcal{A})\ \exists\, b\in\operatorname{Cone}(\mathcal{B}) : \|a - b\|_{L^\infty_w} \le \tfrac{r}{1-r} \|b\|_{L^2}
        \,.
    \end{equation}
    Now, let $a\in\operatorname{Cone}(\mathcal{A})$ be arbitrary and note that, by assumption, there exists some $b\in\operatorname{Cone}(\mathcal{B})$ such that
    \begin{equation}
    \label{eq:simple_ab_dsi}
        \|a -b\|_{L^\infty_w} \le r\|a\|_{L^2}
        \,.
    \end{equation}
    Moreover, it holds that
    $$
        \|a\|_{L^2}
        \le \|a - b\|_{L^2} + \|b\|_{L^2}
        \le \|a - b\|_{L_w^\infty} + \|b\|_{L^2}
        \le r\|a\|_{L^2} + \|b\|_{L^2} \,,
    $$
    resulting in the bound $\|a\|_{L^2} \le \frac{1}{1-r}\|b\|_{L^2}$.
    Inserting this bound into~\eqref{eq:simple_ab_dsi} yields $
        \|a -b\|_{L^\infty_w}
        \le r\|a\|_{L^2}
        \le \tfrac{r}{1 - r} \|b\|_{L^2}
    $ and proves the claim.
\end{proof}

We are now ready to prove the adaptation of Theorem~5.12 in~\cite{TrunschkeEigelNouy2025}.
\begin{theorem}
\label{thm:RIP_MRromega}
    Suppose that $\tfrac{\ubar{\bm\omega}}{\tilde{\bm\omega}} \in\ell^\infty$ and $\tfrac{\tilde{\bm\omega}^2}{\bm\omega} \in \ell^1$.
    Moreover, for $c, r > 0$ define 
    $\tilde{r} := (1+c)^2 \|\tfrac{\ubar{\bm\omega}}{\tilde{\bm\omega}}\|_{\ell^\infty}^4 \|\tfrac{\tilde{\bm\omega}^2}{\bm\omega}\|_{\ell^1}^2 r$.
    Then it holds that
    \begin{equation}
        d_{\mathrm{si}L^\infty_w}(\mathcal{M}^2_{R,r,\bm\omega}, B_{\ell^0_{\tilde{\bm\omega}}}(0,\tilde{r})) \le \frac{1}{c} .
    \end{equation}
\end{theorem}
\begin{proof}
    Let $A\in \mathcal{M}^2_{R,r,{\bm\omega}}$ and $c_1 := \|\frac{\ubar{\bm\omega}}{\tilde{\bm\omega}}\|_{\ell^\infty}$, which is finite because $\frac{\ubar{\bm\omega}}{\tilde{\bm\omega}}\in\ell^1$.
    Since $\|L_{\bm\nu}\|_{L^\infty_w} = \ubar{\bm\omega}_{\bm\nu} \le c_1\tilde{\bm\omega}_{\bm\nu}$, Corollary~2.8 in~\cite{TrunschkeEigelNouy2025} ensures that for every $\tilde{r}>0$ there exists $B \in B_{\ell^0_{\tilde{\bm\omega}}}(0, \tilde{r})$ such that
    \begin{equation}
    \label{eq:bound_L2L_inf_ell05omega}
        \|A - B\|_{L_w^\infty}
        \le \|A - B\|_{\ell^1_{c_1\tilde{\bm\omega}}}
        \le \tilde{r}^{-1/2} \|A\|_{\ell^{2/3}_{(c_1\tilde{\bm\omega})^2}}
        = c_1^2 \tilde{r}^{-1/2} \|A\|_{\ell^{2/3}_{\tilde{\bm\omega}^2}}.
    \end{equation}
    The final $\ell^{2/3}_{\tilde{\bm\omega}^2}$-norm can be bounded by Lemma~2.11 in~\cite{TrunschkeEigelNouy2025},
    yielding
    \begin{equation}
    \label{eq:bound_ell05omega_ell2omega}
        \|A\|_{\ell^{2/3}_{\tilde{\bm\omega}^2}}
        \le \|\tfrac{\tilde{\bm\omega}^2}{\bm\omega}\|_{\ell^1} \|A\|_{\ell^2_{\bm\omega}} \,.
    \end{equation}
    Since $A\in\mathcal{M}^2_{R,r,\bm\omega}$ can be written as $A=QC$ for some $\bm\omega^2$-orthogonal $Q$ and sparse $C$, it follows that
    \begin{equation}
    \label{eq:bound_ell2omega_ell2beta}
        \|{A}\|_{\ell^2_{\bm\omega}}^2
        = (QC)^\intercal \operatorname{diag}({\bm\omega}^2) (QC)
        = C^\intercal \operatorname{diag}(\beta_Q^2) C
        = \|{C}\|_{\ell^2_{\beta_Q}}^2
    \end{equation}
    for some $\bm\beta_Q$, depending on $\bm\omega$ and $Q$.
    Now let $\mathcal{S} := \operatorname{supp}(C)$ and bound
    \begin{equation}
    \label{eq:bound_ell2beta_ell2}
        \|C\|_{\ell^2_{\bm\beta_Q}}^2
        = \sum_{\bm{m}\in \mathcal{S}} \bm \beta_{Q,\bm m}^2 C_{\bm m}^2
        \le \sum_{\bm m\in \mathcal S} \Big(\sum_{\bm n\in \mathcal S}\bm \beta_{Q,\bm n}^2\Big) C_{\bm m}^2
        = \|C\|_{\ell^0_{\bm\beta_Q}} \|{C}\|_{\ell^2}^2
        \le r \|C\|_{\ell^2}^2
        = r \|A\|_{\ell^2}^2 .
    \end{equation}
    Combining equations~\eqref{eq:bound_L2L_inf_ell05omega},~\eqref{eq:bound_ell05omega_ell2omega},~\eqref{eq:bound_ell2omega_ell2beta} and~\eqref{eq:bound_ell2beta_ell2} results in the bound
    \begin{equation}
        \|A - B\|_{L_w^\infty}
        \le \|\tfrac{\ubar{\bm\omega}}{\tilde{\bm\omega}}\|_{\ell^\infty}^2 \|\tfrac{\tilde{\bm\omega}^2}{\bm\omega}\|_{\ell^1} \sqrt{\tfrac{r}{\tilde{r}}} \|{A}\|_{L^2} \,.
    \end{equation}
    Now, recall that $\tilde{r} = (1 + c)^2 \|\tfrac{\ubar{\bm\omega}}{\tilde{\bm\omega}}\|_{\ell^\infty}^4 \|\tfrac{\tilde{\bm\omega}^2}{\bm\omega}\|_{\ell^1}^2 r$ and, therefore,
    \begin{equation}
        \|A - B\|_{L_w^\infty}
        \le \tfrac{1}{1+c} \|{A}\|_{L^2} \,.
    \end{equation}
    Lemma~\ref{lem:simplified_dsi} now implies $d_{\mathrm{si}L^\infty_w}(A, B) \le \frac{1}{c}$.
\end{proof}

Applying Theorem~\ref{thm:RIP_MRromega} with $\tilde{\bm\omega} = \ubar{\bm\omega}$ implies
\begin{equation}
    d_{\mathrm{si}L^\infty_w}(\mathcal{M}^2_{R,2r(\gamma),\bm\omega}, B_{\ell^0_{\ubar{\bm\omega}}}(0,\tilde{r})) \le \frac{1}{c}
\end{equation}
with $\tilde{r} := (1+c)^2 \|\tfrac{\ubar{\bm\omega}^2}{\bm\omega}\|_{\ell^1}^2 2r(\gamma)$.
Now, suppose that $\delta\le\frac{1}{2}$ and $c\ge\frac{15}{\delta}$, then 
Theorem~5.9 in~\cite{TrunschkeEigelNouy2025} ensures $\operatorname{RIP}_{B_{\ell^0_{\bm\omega}}(0, \tilde{r})}(\tfrac{\delta}2)$ implies $\operatorname{RIP}_{\mathcal{M}^2_{R,2r(\gamma),\bm\omega}}(\delta)$.

It, thus, remains to ensure $\operatorname{RIP}_{B_{\ell^0_{\bm\omega}}(0, \tilde{r})}(\tfrac{\delta}2)$.
By Theorem~1.5 in~\cite{TrunschkeEigelNouy2025}, there exists a constant $C>0$ such that this can be achieved with high probability for
\begin{equation}
    n = 4 C \delta^{-2} \tilde{r}\max\{M\log^3(\tilde{r})\log(N), \log(p^{-1})\} \,.
\end{equation}
Plugging in $c=\frac{15}\delta$ and removing the logarithmic factors (as in the proof of Theorem~\ref{thm:sparse_sample_complexity}) yields
\begin{align}
    n
    &\ge \tilde{C} C \delta^{-4} r(\gamma)^{1+\varepsilon} \log(p^{-1}) \\
    &= \tilde{C} C \delta^{-4} \gamma^{-1/s+\varepsilon} \log(p^{-1}) \,.
\end{align}